\documentclass[12pt,reqno]{amsart}

\usepackage[mathlines]{lineno}

\usepackage{times}
\usepackage{amsfonts}
\usepackage{amssymb}
\usepackage{latexsym}
\usepackage{xcolor}

\newtheorem{theorem}{Theorem}
\newtheorem{lemma}[theorem]{Lemma}
\newtheorem{prop}[theorem]{Proposition}
\newtheorem{definition}[theorem]{Definition}
\newtheorem{remark}{Remark}
\newtheorem{corollary}[theorem]{Corollary}
\newtheorem{example}{Example}

\begin{document}

\begin{center}
\Large{On a General Concept of a  Hausdorff-Type Operator}
\end{center}

\centerline{A. R. Mirotin}

\centerline{amirotin@yandex.ru}

\

\textsc{Abstract.} A general approach to the concept of a Hausdorff operator is proposed in such a way that  a number of classical and new operators fit into the given definition.  Conditions  are given for the boundedness of
the operators under consideration  in  Banach function spaces,  Haj{\l}asz-Sobolev spaces, and in the atomic Hardy space  $H^1$,   and their regularity property  is investigated. 
 Examples and special cases  are considered. This approach is aimed  to unify the study of a lot of extensions and analogs of the classical 
Hausdorff operator.


\

\

Key words and phrases. Hausdorff operator, isomorphism, automorphism,  locally compact group,  Banach function space,  Haj{\l}asz-Sobolev space,  Hardy space, space of homogeneous   type, homogeneous space of a locally compact group, double coset space.

\

\

\section{ Introduction}\setcounter{equation}{0}

Hausdorff  operator on the semi-axis  was introduced by Garabedian and Rogosinskii in the form
$$
Hf(x)=\int_0^1f(ux)d\mu(u),\ x\in (0,\infty),
$$
where $\mu$ stands for a finite measure on $(0,1]$.
This is  a natural continuous analog of the Hausdorff summation method  (see \cite[Chapter XI]{H}). Here a map $x\mapsto ux$ ($u>0$) is an automorphism of the additive topological semigroup $(0,\infty)$.

After episodic appearances of several  one-dimensional Hausdorff
operators (see, e.g., \cite{X1} and \cite{X2}), the impetus
 for the modern development of this theory was given by the work of Liflyand and  M\'{o}ricz \cite{LM} where   Hausdorff  operators on one-dimensional Hardy space were considered.
The first nontrivial
results in several dimensions appeared  in  \cite{LL} for operators of the form
$$
Hf(x)=\int_{\mathbb{R}^m}\Phi(u)f(A(u)x)du.
$$ 
Here $f:\mathbb{R}^n\to \mathbb{C}$, $A(u)\in {\rm GL}(n,\mathbb{R})$ for a.~e. $u\in \mathbb{R}^m$, $\Phi\in L^1_{\rm loc}(\mathbb{R}^m)$ ($x\in \mathbb{R}^n$ is a column vector). In this case,  the map $x\mapsto A(u)x$  is an automorphism of the additive topological group $\mathbb{R}^n$ if $A(u)\in {\rm GL}(n,\mathbb{R})$.

More information on  the first period of the development of the theory of Hausdorff
operators one can find in the survey   articles  \cite{Ls}, \cite{CFW}.

In resent two  decades different notions of a Hausdorff operator have been suggested (see, e.~g., \cite{LL, Ls, CFW, KL, JMAA, SG, S, MCAOT, MCAOT2, Lie, JOTH} and bibliography therein).
The aim of this work is to work out a general form of a Hausdorff-type operator that includes the types of operators in the aforementioned works as well as some other classes of operators as special cases. The key word in our approach is ``an automorphism''. \footnote{ At the first time, this notion in the definition of a Hausdorff operator appeared in \cite{JMAA}.}
Thus, in our opinion, the unified   approach to the notion of a Hausdorff operator may be as follows.

 Let  $\frak{R}$ and $\frak{S}$ be two sets which are  an isomorphic  objects of some category $\mathfrak{K}$. For instance  $\frak{R}$ and $\frak{S}$ are  endowed with some mathematical structure  (algebraical, topological, analytical, algebraical-topological, order, measure,  etc.) in a sense of N. ~Bourbaki. 
Let $\mathrm{Isom}(\frak{R},\frak{S})$ stands for  the set of all isomorphisms  $\frak{R}\to\frak{S}$ in this category, and $(\Omega,\mu)$ denotes some measure space. Finally, let $A: \Omega\to \mathrm{Isom}(\frak{R},\frak{S})$ be some measurable map (in a sense which will be specified  in each concrete situation,   see, e.g.,  definitions \ref{egree}, \ref{measure}, \ref{mes_Iso}, \ref{munu}, and \ref{measAut} below) defined a.~e. $[\mu]$, and $\Phi$ a given $\mu$-measurable function on $\Omega$.

\begin{definition}\label{general1}
   Let the assumptions mentioned above hold.  {\it A Hausdorff operator} (or a Hausdorff-type operator) acts on a  functions $f:\frak{S}\to V$ (here $V$ is some topological vector space) by the rule
\begin{equation}\label{general}
(\mathcal{H}_{\Phi,A,\mu}f)(x) =(\mathcal{H}_{\Phi,A}f)(x) =\int_{\Omega} \Phi(u)f(A(u)(x))d\mu(u), \ \ x\in \frak{R},
\end{equation}
provided the integral converge in a suitable sense.
\end{definition}

In the rest of the paper we shall consider the case $\frak{R}=\frak{S}$ and denote by $\mathrm{Aut}(\frak{S}):=\mathrm{Isom}(\frak{S},\frak{S})$  the set of all automorphisms of $\frak{S}$ in this category.

The notion of  a Hausdorff-type operator contains as special cases the classical and discrete Hilbert transforms and there generalizations, Hankel operators, Hausdorff  and convolution operators on groups, Hausdorff-Zhu   operators, etc. (see examples below). The author hopes that this approach will allow one to unify the study of a lot of extensions and analogs of the classical 
Hausdorff operator.

The article is organized as follows. Section \ref{Special cases} contains examples that show that a number of known  and new operators fit into Definition \ref{general1}.
In Section \ref{regularity} we generalize the   regularity result of Rogosinskii and Garabedian obtained by this authors for the classical one-dimensional case. Sections  \ref{sec:kote},  \ref{sec:sobolev},
and \ref{H10} are devoted to   Hausdorff-type operators in Banach function spaces, general  Sobolev spaces, and general Hardy spaces respectively. 
We obtain sufficient conditions for the boundedness of Hausdorff-type operators in these spaces if the mathematical structures involved are consistent in a natural way.  
The section \ref{H10} consists of  four subsections.  The main result of the first subsection is  on the boundedness of a  Hausdorff-type operators on general Hardy spaces $H^1$ (Theorem \ref{H1}).  In the remaining three subsections, Theorem \ref{H1} is applied to important cases of operators on locally compact groups, homogeneous spaces of Lie groups, and double coset  spaces of Lie groups. In this cases some assumptions of this theorem   hold automatically. 
\footnote{  It should be noted that in this cases our conditions of boundedness differ from the conditions given in \cite{JMAA, Lie, JOTH}.}

\section{Special cases}\label{Special cases}\setcounter{equation}{0}

\begin{example}\label{det}  (Hausdorff-type operator over a matrix algebra.) {\rm Let $\frak{S}=\mathrm{Mat}_n(\frak{k})$ be the algebra  of square matrices $M=(m_{ij})$ of order $n$ over a field $\frak{k}$. Every $M\in \mathrm{Mat}_n(\frak{k})$ has the form $M=(m_1,\dots,m_n)$, where $m_j$ stands for the $j$th column of $M$. For each permutation $\sigma\in\mathbf{S}_n$ ($\mathbf{S}_n$ denotes the symmetric group of order $n$) we  denote by $A(\sigma)$ the bijection of  $\mathrm{Mat}_n(\frak{k})$ (automorphism in the category of sets) such that $A(\sigma)(M)=(m_{\sigma(1)},\dots, m_{\sigma(n)})$. We equip the set $\Omega= \mathbf{S}_n$ with the counting measure. Let  $\Phi(\sigma)=\mathrm{sgn}(\sigma)$, where $\mathrm{sgn}( \sigma):=1$ if $\sigma$  is even, and  $\mathrm{sgn}( \sigma):=-1$ otherwise. 
Then a Hausdorff operator in the sense of Definition \ref{general1} acts on the function $f:\mathrm{Mat}_n(\frak{k})\to V$ as
$$
(\mathcal{H}_{\Phi,A}f)(M) =\sum_{\sigma\in \mathbf{S}_n}\mathrm{sgn}( \sigma)f(A(\sigma)(M)).
$$
In particular, if we take $\frak{k}=V=\mathbb{C}$,  $f_0(M)=\prod_{i=1}^n m_{ii}$, then
\begin{eqnarray*}
(\mathcal{H}_{\Phi,A}f_0)(M) &=&\sum_{\sigma\in \mathbf{S}_n}\mathrm{sgn}(\sigma)f_0(A(\sigma)(M))\\
&=&\sum_{\sigma\in \mathbf{S}_n}\mathrm{sgn}(\sigma)f_0(A(\sigma)(m_1,\dots,m_n)).
\end{eqnarray*}
Since the right-hand side here is an alternate multilinear   form (as a function of column vectors $m_1,\dots,m_n$), we have  
$$
(\mathcal{H}_{\Phi,A}f_0)(M) =\mathrm{det}(M)
$$
(see also \cite[p. 202]{Berb}).

One can take also as $\frak{S}$ any  subset of $\mathrm{Mat}_n(\frak{k})$ which is invariant with respect to some family  of automorphisms $(A(\sigma))_{\sigma\in \Sigma}$.
}
\end{example}

\begin{example}\label{DiscreteHilbert}(The Discrete Hilbert transform.) {\rm Let $\frak{S}=\mathbb{Z}$ be the ring of integers  with its natural order, $\Omega=\mathbb{Z}$ endowed with a discrete  measure $\mu(\{k\})=p_k$, and $A(u)(k)=k-u$  ($k,u\in  \mathbb{Z}$) an order preserving bijections of  $\mathbb{Z}$ (automorphisms in the category of linearly ordered sets).
Let
$$\Phi(u)=
\begin{cases}
\frac{2}{\pi u}, u\mbox{  odd}\\
0, u \mbox{  even}.
\end{cases}
$$
 In this case, \eqref{general} takes the form
 $$
 (Hf)(k)=\sum_{u\in \mathbb{Z}}\Phi(u)f(k-u)p_u=
 \begin{cases}
\frac{2}{\pi}\sum\limits_{n \  \mathrm{odd}}\frac{f(n)}{k-n}p_{k-n},  k \mbox{  even}\\
\frac{2}{\pi}\sum\limits_{n \  \mathrm{even}}\frac{f(n)}{k-n}p_{k-n},  k \mbox{  odd}.
\end{cases}
 $$
For $p_k\equiv 1$ this is the discrete Hilbert transform of a   function  $f:\mathbb{Z}\to V$ (for the case $V=\mathbb{C}$ see \cite{Kak}).
}
\end{example}

\begin{example}\label{Hilbert}(The Hilbert transform.) {\rm Let $\frak{S}=\Omega$ be the real line $\mathbb{R}$ with Euclidean metric and Lebesgue measure, $A(u)(x)=x-u$  ($x,u\in  \mathbb{R}$) a distance preserving bijections of  $\mathbb{R}$,
and $\Phi(u)=\frac{1}{\pi u}$, the Cauchy kernel. In this case, \eqref{general} takes the form
$$
(\mathrm{H}f)(x)=\frac{1}{\pi }\mbox{p.v.}\int_{-\infty}^\infty \frac{f(x-u)}{u}du,
$$
the Hilbert transform of a measurable  function  $f:\mathbb{R}\to V$.

 Calder\'{o}n-Zygmund operators can be considered  in a similar manner.

}
\end{example}

The previous example can be generalized in the following way.

\begin{example}\label{GHilbert}(The Hilbert transform along curves,  see \cite{SW, NRW1, NRW3}).  {\rm Let $\frak{S}=\mathbb{R}^n$ with Euclidean metric, $\Omega=\mathbb{R}$ with Lebesgue measure, and $A(u)(x)=x-\gamma(u)$  ($x\in  \mathbb{R}^n$) a distance preserving bijections of  $\mathbb{R}^n$ where $\gamma : \mathbb{R}\to \mathbb{R}^n$ is a suitable function  (say
polynomial) satisfying $\gamma(0) = 0$. Then the singular integral operator
$$
(\mathrm{H}f)(x) :=\int_{-\infty}^\infty \Phi(u) f(x - \gamma(u)) du,
$$
where $\Phi$ is a Calder\'{o}n-Zygmund  kernel, is of the form \eqref{general}.
}
\end{example}

\begin{example}\label{Graph}(Hausdorff-type operators on graphs.)  {\rm Let $\mathcal{G}$ be a countable infinite graph, $\frak{S}=W$ the set of vertices of $\mathcal{G}$ with the counting measure.  Recall that an  automorphism of $\mathcal{G}$ is a bijection $\gamma:W\to W$ such that two vertices $v_1$ and $v_2$ are connected by a wedge if and only if $\gamma(v_1)$ and $\gamma(v_2)$ are connected by a wedge. Let $\Gamma$ be some countable subgroup of the automorphism group $\mathrm{Aut}(\mathcal{G})$  endowed with its natural topology (see \cite{trof}).Then   a Hausdorff operator on $\mathcal{G}$ acts on a   function $f:W\to \mathbb{C}$ as follows
$$
(Hf)(x)=\sum_{\gamma\in \Gamma}c_{\gamma}f(\gamma(x)) 
$$
(provided the series converge), where $c_{\gamma}\in \mathbb{C}$.
An averaging function of a periodic graph has this form (see, e.~g., \cite{Lenz-Pogorzelski-Schmidt}). 
}
\end{example}

From now on      we shall assume that the integral in \eqref{general} exists in the sense of Lebesgue.

\begin{example}\label{Hankel}  (Hankel integral operators.) Let  {\rm $\frak{S}=\mathbb{R}$ with standard metric, $\Omega=\mathbb{R}$ endowed with Lebesgue measure,
and $A(u)(x)=u-x$ ($u\in \mathbb{R}$) the automorphisms of the metric space $\mathbb{R}$. 
For $\varphi\in L^1_{loc}(\mathbb{R})$ consider the Hausdorff-type operator
\begin{align*}
(H_\varphi f)(x)=\int_{\mathbb{R}}\varphi(u)f(u-x)du.
\end{align*}
If we consider  measurable  functions $f$ on  $\mathbb{R}$ supported in $\mathbb{R}_+$ then the operator
\begin{align*}
(\Gamma_\varphi f)(x):=(H_\varphi f)(x)=\int_{\mathbb{R}}\varphi(x+t)f(t)dt=\int_{\mathbb{R}_+}\varphi(x+t)f(t)dt
\end{align*}
 is a Hankel integral operator (see, e.~g., \cite[p. 46]{Pel}).
}
\end{example}

\begin{example}\label{cauchi} (A Hausdorff transform over a circular manifold.) {\rm  Let    $\frak{S}$ be  a complex  submanifold of $\mathbb{C}^n$ with automorphisms  (biholomorphic mappings)  $A(u)(z)=(u_1z_1,\dots, u_n z_n)$, where $u=(u_1,\dots,u_n)\in \mathbb{T}^n$ (e.~g., let $\frak{S}$ be  a Reinchart domain  in $\mathbb{C}^n$, or the torus  $ \mathbb{T}^n$). Let also
 $\Omega= \mathbb{T}^n$ endowed  with the Lebesgue measure,  and 
 $$
 \Phi(u)=\frac{1}{(2\pi \imath)^n}\frac{1}{(u_1-a_1)\dots (u_n-a_n)} 
 $$
(for simplicity we assume that $a_j\notin \mathbb{T},  j=1,\dots,n$).   In this case, \eqref{general} turns into the operator
$$
(\mathcal{H}f)(z)=\frac{1}{(2\pi \imath)^n}\int_{ \mathbb{T}^n}\frac{f(u_1z_1,\dots, u_n z_n)}{(u_1-a_1)\dots (u_n-a_n)}du_1\dots du_n,\  z\in \frak{S}
$$
for a measurable complex function $f$ on  $\frak{S}$. If  $\frak{S}= \mathbb{T}^n$, then  putting $u_j=\zeta_j/z_j$  for $ j=1,\dots,n$ we get
$$
(\mathcal{H}f)(z)=\frac{1}{(2\pi \imath)^n}\int_{ \mathbb{T}^n}\frac{f(\zeta_1,\dots, \zeta_n)}{(\zeta_1-a_1z_1)\dots (\zeta_n-a_nz_n)}d\zeta_1\dots d\zeta_n,\ z\in \mathbb{T}^n.
$$
}
\end{example}

\begin{example}\label{conv}(A convolution with a measure.) {\rm Let $\frak{S}=G=\Omega$ be a multiplicative group equipped with some (left) invariant metric,  $\mathrm{Aut}(\frak{S})= \mathrm{Iso}(G)$ the set of isometries of $G$ (automorphisms in the category of metric spaces),  $A(u)(x)=u^{-1}x$ ($u\in G$), $\Phi(u)=1$. In this case, \eqref{general} turns into a  convolution operator $f\mapsto f\ast\mu$ on $G$.}
\end{example}

\begin{example}\label{Kothe}  (Hausdorff-type operators over measure spaces.)  {\rm Let $(\frak{S}, \mathcal{B}, \nu)$ be a  measure space. 
By an  
automorphism of a space $\frak{S}$ we mean a bijective  map $\tau:\frak{S}\to \frak{S}$ 
 such that both  $\tau$ and  $\tau^{-1}$ are measurable and $\nu(\tau(E))=\nu(E)$ for every $E\in  \mathcal{B}$. Let $X$ be a   function space on $\frak{S}$
 which is 
invariant under a family $(\tau(u))_{u\in \Omega}$ of automorphisms of $\frak{S}$  (i.~e. $f\in X\Rightarrow f\circ \tau(u)^{-1}\in X\forall u\in\Omega$, see, e.~g., \cite{LiTz}). Hausdorff  operator on $X$ takes the form
\begin{equation}\label{generalKothe}
(\mathcal{H}_{\Phi,\tau}f)(x) =\int_{\Omega} \Phi(u)f(\tau(u)^{-1}(x))d\mu(u),\ \ f\in X.
\end{equation}
 Theorem \ref{Kote}  below gives conditions for boundedness of this operator in some Banach function spaces.

The following  very special case of an operator of the form \eqref{generalKothe} arises in the ergodic theory:
$$
(Hf)(x)=\frac 1t\int_0^tf(T(u)(x))du.
$$
Here $\Omega=[0,t]$, $\mu$ is the Lebesgue measure on $[0,t]$, $\Phi(u)\equiv 1/t$, and $T(u)$ stands for some one-parameter group of  automorphisms of a  measure space $\frak{S}$ (a flow).
}
\end{example}

\begin{example}\label{topgroup}  (Hausdorff operator over a topological group.) {\rm  Let $\frak{S}=G$ be a topological group,  $\mathrm{Aut}(\frak{S})=\mathrm{Aut}(G)$ the group of all topological automorphisms of $G$. In this  case, we have  got a general definition of  a Hausdorff operator on topological groups.  This example contains several known definitions of Hausdorff operators for classical groups (see subsection \ref{group},  \cite{JMAA}, \cite{Nachr}, and examples therein).

For example, the  Harish-Chandra transform leads to a Hausdorff operator on the group $\mathrm{SL}(2,\mathbb{R})$.
In fact, one of the form of the Harish-Chandra transform for   $\mathrm{SL}(2,\mathbb{R})$  looks as  $H^Kf(x)=2\imath\sin\theta(x)(\mathbf{H}f)(x)$  where the matrix $x\in \mathrm{SO}(2,\mathbb{R})$ represents the rotation of the plane by an angle $\theta(x)$. Here $\mathbf{H}$ stands for a Hausdorff  operator
$$
(\mathbf{H}f)(x)=\int_{A^+}\frac{\alpha(u)+\alpha(u^{-1})}{2}f(u^{-1}xu)d\mu(u),
$$
 where $A^+$ denotes the set of $2\times 2$ matrices of the form $u=\mathrm{diag}(a,a^{-1})$, $a\ge 1$, $\alpha(u)=a^2$, and $d\mu(u)=da$  (see, e.g., \cite[Chapter VII, \S 5]{Leng}).
}
\end{example}

\begin{example}\label{delsarte} (Generalized   shift operator of Delsarte \cite{Delsarte}, \cite[Chapter I, \S 2]{Lev}). {\rm Let $G$ be a topological group and $\Omega$  a compact subgroup of $\mathrm{Aut}(G)$ with normalized Haar measure $\mu$. The generalized   shift operator of Delsarte is as follows
$$
T^xf(h)=\int_{\Omega} f(hu(x))d\mu(u)\quad (x,h\in G).
$$
This is a composition of the usual shift $f(x)\mapsto f(hx)$ on $G$ and a Hausdorff operator $g\mapsto T^{(\cdot)}g(e)$ over $G$ in a sense of the example \ref{topgroup} (in this case  $\Phi(u)\equiv 1$, $A(u)=u$).}
\end{example}

\begin{example}\label{discrete} (Discrete Hausdorff operator over the Euclidean space.) {\rm  Let $\frak{S}=\mathbb{R}^d$ be a $d$-dimensional Euclidean space considered as an additive topological group. Then   the group $\mathrm{Aut}(\mathbb{R}^d)$ of all topological automorphisms of $\mathbb{R}^d$ can be identified with the general linear group $\mathrm{GL}(d,\mathbb{R})$. Let $\Omega=\mathbb{Z}$ be endowed with the counting measure. In this case, \eqref{general} turns into a so-called discrete Hausdorff operator
\begin{equation*}
(\mathcal{H}_{\Phi,A}f)(x) =\sum_{k\in \mathbb{Z}} \Phi(k)f(A(k)x),
\end{equation*}
where $A(k)\in \mathrm{GL}(d,\mathbb{R})$, $x\in \mathbb{R}^d$ is a column vector. For the spectral theory of such operators see \cite{JMS, MMAS}.

For example, in quantum mechanics  the expansion of the wave function $\Psi$ into components for a three body system under some conditions takes the form
$$
\Psi(x)=f(x)+f(P^+x)+f(P^-x),\ x\in \mathbb{R}^3
$$
where $P^{\pm}$ denote matrices of operators of cyclic permutation of particles and $f$ stands for the first component of the   wave function (see \cite[Formula (7.24)]{MF}).

Discrete Hausdorff operators appeared also in analysis, see, e.~g.,  \cite{makarov}, and  functional differential equations, see, e.~g.,   \cite{RSPA}, \cite{Ross},    \cite{Ross2}, \cite{Zaidi-Van Brunt-Wake}.
}
\end{example}

\begin{example}\label{homogen}  (Hausdorff-type operator over a homogeneous space.) {\rm Let $\frak{S}=G/K$ be a homogeneous space of a  locally compact group $G$, $K$ a compact subgroup of $G$.  In this case, $\mathrm{Aut}(\frak{S})$ can be identified with some factor-group of  the group  $\mathrm{Aut}_K(G)$ of all topological automorphisms of $G$ which map $K$ onto itself and we get a Hausdorff operator over a homogeneous space (see subsection \ref{homogen} below and \cite{homogen}, \cite{Lie}).}
\end{example}

\begin{example}\label{doubl}  (Hausdorff-type operator over  a double coset space of a  topological group.)  {\rm Let $\frak{S}=G/\!/K$ be a double coset space of a  locally compact group $G$, $K$ a  compact subgroup of $G$.  In this case again $\mathrm{Aut}(\frak{S})$ can be identified with  some factor-group of  the group  $\mathrm{Aut}_K(G)$ of all topological automorphisms of $G$ which map $K$ onto itself  (see subsection \ref{double} below and \cite{JOTH} for details).}
\end{example}

\begin{example}\label{disc}   (Hausdorff-type operators over  the unit disc         and over the ball in $\mathbb{C}^n$.)  {\rm Let $\frak{S}=\mathbb{D}$ be the unite disc in the complex plane with its natural analytic structure, $\mathrm{Aut}(\frak{S})=\mathrm{Aut}_0(\mathbb{D})$ the group of all involutive  M${\ddot {\rm o}}$bius automorphisms of $\mathbb{D}$, 
$$
A(u)(z)=\frac{u-z}{1-\bar{u}z}, \  \ u\in \Omega:=\mathbb{D}. 
$$
In this case, \eqref{general} turns into a so-called Hausdorff-Zhu operator in the disc (see \cite{BONET, MCAOT,  MCAOT2, KM, KGM}).

 A similar construction works if $\frak{S}$ is the unit ball in $\mathbb{C}^n$ \cite{KarMirMMAS}.
 
 }
\end{example}

\begin{example}\label{Cplus}  (Hausdorff operators over  the  half-plane or the whole complex plane.)
 {\rm Let $\mathbb{C}^+$ be the upper half-plane of the complex plane with its natural analytic structure, $\frak{S}=(\mathbb{C}^+)^n$ , $\Omega=(0,\infty)^n$, and $A(u)(z)=(\frac{z_1}{u_1},\dots,\frac{z_n}{u_n})$ ($u\in\Omega$) a biholomorphic map of $(\mathbb{C}^+)^n$. In the case $n=1$, \eqref{general} turns into a so called Hausdorff operator over the  upper half-plane (see \cite{BONET, S, SG, Hung} and the bibliography therein). 
 One can to generalize this example using the whole group of  M${\ddot {\rm o}}$bius automorphisms of the upper half-plane.
 
  A similar approach works for the  Fock space and other spaces of entire functions (see \cite{blasco, BONET, SG} and the bibliography therein).}
\end{example} 

\begin{example}\label{Hecke}  (Hecke operator.)   {\rm Let $\mathbb{C}^+$ be the upper half-plane  with its natural analytic structure, $\frak{S}=\mathbb{C}^+$, $m$ and $k$ are natural numbers,
$$
\Omega=\{u\in \mathbb{N} \times\mathbb{Z_+} \times\mathbb{N}: u= (a,b,d),  ad=m, 0\le b<d\}
$$
with the counting measure, $A(u)(z)=(az+b)d^{-1}$ for $z\in \mathbb{C}^+$, $\Phi(u)=m^{k-1}d^{-k}$  ($u\in \Omega$). 
Then  \eqref{general} turns into the Hecke operator
$$
T_mf(z)=m^{k-1}\sum_{(a,b,d)\in\Omega}d^{-k}f\left(\frac{az+b}{d}\right)
$$
(in number theory  $f$ stands for a modular form of weight $k$, see, e.~g., \cite{Serre}).
}
\end{example}

\

\section{The regularity property of Hausdorff-type operators}\label{regularity}\setcounter{equation}{0}

To examine the  regularity property of the transformation $\mathcal{H}_{\Phi,A}$ we need the following definition. 

\begin{definition}\label{filter}
 Let $\mathcal{F}$ be a filter  on a set $\frak{S}$. We say that  a family $(A(u))_{u\in \Omega}$ of automorphisms of $\frak{S}$
\textit{agrees with the filter } $\mathcal{F}$  if  $A(u)^{-1}(B)$ belongs to $\mathcal{F}$   for each $B\in \mathcal{F}$ and  every $u\in \Omega$.
\end{definition}

The next proposition is a wide generalization of the classical result of Gerabedyan and Rogosinskii (see \cite[Chapter XI]{H}).

\begin{prop}\label{prop2}
Suppose that the conditions of Definition \ref{general1} are fulfilled and a filter  $\mathcal{F}$ on $\frak{S}$ has a countable base. Let a family $(A(u))_{u\in \Omega}$ of automorphisms of $\frak{S}$ agrees with  $\mathcal{F}$ and
\begin{equation}\label{norm}
\int_{\Omega} \Phi(u)d\mu(u)=1.
\end{equation}
Then the transformation $\mathcal{H}_{\Phi,A}$  is
regular in the following sense. For every Banach space $V$ and for every bounded   function $f:\frak{S}\to V$ such that the function $u\mapsto f(A(u)(x))$ is $\mu-$measurable for each $x\in \frak{S}$ the equality  $\lim_{x,\mathcal{F}} f(x) = l$    implies $\lim_{x,\mathcal{F}}(\mathcal{H}_{\Phi,A} f)(x) = l$.

\end{prop}

\begin{proof}

Let  $\lim_{x,\mathcal{F}} f(x) = l$. Then 
\begin{equation}\label{lim}
\lim_{x,\mathcal{F}} f(A(u)(x)) = l
\end{equation}
for all $u\in \Omega$. Indeed, for every $\varepsilon>0$ there exists such $B_\varepsilon \in \mathcal{F}$ that $\|f(y)-l\|<\varepsilon$ for all $y\in B_\varepsilon$. It follows that $\|f(A(u)(x))-l\|<\varepsilon$ for all $x\in A(u)^{-1}(B_\varepsilon)$, as well. By the Definition \ref{filter} we have  $A(u)^{-1}(B_\varepsilon)\in \mathcal{F}$   for each  $u\in \Omega$ and \eqref{lim} follows.  Now by the Lebesgue Theorem for the Bochner integral (one can apply the Lebesgue Theorem, since the  base of $\mathcal{F}$ is  countable) one has
$$
\lim_{x,\mathcal{F}}(\mathcal{H}_{\Phi,A} f)(x) =\int_\Omega \Phi(u)ld\mu(u)= l.
$$
\end{proof}

\begin{example}\label{topgroup3} {\rm  Let in the Example \ref{topgroup}   $G=\cup_{n=1}^\infty K_n$ be sigma-compact ($K_n$ is an increasing sequence of compact subsets of $G$) with the Haar measure $\mu$ and  $\mathcal{F}$ be the filter whose base consists of  all complements  $G\setminus  K_n$ where $n\in\mathbb{N}$.        In this case, all the conditions of Definition \ref{filter}      are fulfilled for every topological automorphosm $A$ of $G$, and   $\lim_{x,\mathcal{F}} f(x) = \lim_{x\to\infty} f(x)$. Assume that the family   of topological  automorphisms $(A(u))_{u\in \Omega}$ of $G$ is $\mu$-$\nu$  measurable in a sense of Definition \ref{munu} below.
 Then the Proposition \ref{prop2} implies that under the condition \eqref{norm} one has $\lim_{x\to \infty}(\mathcal{H}_{\Phi,A}f)(x)=l$ whenever $\lim_{x\to \infty}f(x)=l$ for a bounded measurable complex function $f$ on $G$.

}
\end{example}

\section{ Boundedness of Hausdorff-type operators on  rearrangement-invariant Banach function spaces.}
\label{sec:kote}\setcounter{equation}{0}

In this section, $(\frak{S}, \mathcal{B}, \nu)$ stands for a  measure space with a positive measure $\nu$.

\subsection{The case of $L^p$ spaces}\label{Lp0}\setcounter{equation}{0}

To formulate a result on the $L^p$ boundedness of the operator \eqref{general} we need the following notion.
 
\begin{definition}\label{egree} Let $(\frak{S}, \mathcal{B}, \nu)$ be a   measure space. We say that the family $(A(u))_{u\in \Omega}$ of automorphisms of $\frak{S}$
\textit{agrees with the measure} $\nu$   if for each $E\in \mathcal{B}$ of finite measure and for every $u\in \Omega$ we have $A(u)^{-1}(E)\in \mathcal{B}$ and 
\[
\nu(A(u)^{-1}(E))= m(A(u))^{-1}\nu(E)
\]
for some positive $\mu$-measurable function $u\mapsto m(A(u))$.
\end{definition}

\begin{prop}\label{Lp} Let  the measure $\nu$ be   sigma-finite, the family $(A(u))_{u\in \Omega}$ agrees with $\nu$,  and $1\leq p\le \infty$. If 
$$ 
  \|\Phi\|_{A,p}:=\int_\Omega |\Phi(u)|m(A(u))^{-1/p}d\mu(u)<\infty
$$
(here $ \|\Phi\|_{A,\infty}:=\|\Phi\|_{L^1(\mu)}$),  then the operator 
$\mathcal{H}_{\Phi,A}$ is bounded in $L^p(\nu)$ and its norm does not exceed $ \|\Phi\|_{A,p}$.
\end{prop}

\begin{proof}
Using Minkowskii integral
inequality  we have for $1< p<\infty$ and $f\in L^p(\nu)$
\begin{eqnarray*}
\|\mathcal{H}_{\Phi, A}f\|_{L^p(\nu)}&=&\left(\int_{\frak{S}} \left|\int_\Omega \Phi(u)f(A(u)(x))d\mu(u) \right|^pd\nu(x)\right)^{1/p}\\
&\leq&
\int_\Omega\left(\int_{\frak{S}}|\Phi(u)|^p|f(A(u)(x))|^pd\nu(x)\right)^{1/p}d\mu(u)\\
&=&\int_\Omega|\Phi(u)|\left(\int_{\frak{S}}|f(A(u)(x))|^pd\nu(x)\right)^{1/p}d\mu(u).
\end{eqnarray*}
Since  by the Definition \ref{egree}
\begin{equation}\label{IfA(u)}
\int_{\frak{S}}|f(A(u)(x))|^pd\nu(x)= m(A(u))^{-1}\int_{\frak{S}}|f(x)|^pd\nu(x)
\end{equation}
(it suffices to verify the last equality for $f=\chi_E$, the indicator of a $\nu$-measurable set $E\subseteq \frak{S}$ of finite measure),
we have
\begin{eqnarray*}
\|\mathcal{H}_{\Phi, A}f\|_{L^p(\nu)}&\leq&\int_\Omega|\Phi(u)|m(A(u))^{-1/p}d\mu(u)\left(\int_{\frak{S}}|f(x)|^pd\nu(x)\right)^{1/p}\\
&=&
\|\Phi\|_{A,p}\|f\|_{L^p(\nu)}.
\end{eqnarray*}
For $p=1$ the statement of the proposition follows from Fubini Theorem and for $p=\infty$ it is obvious.
\end{proof}

\begin{example}\label{topgroup2} {\rm  Let in the Example \ref{topgroup}  $G$ be locally compact,    $\nu$  the Haar measure of $G$, and  the   family $\mathrm{Aut}(G)$  of all topological automorphisms of a  group  $G$ is equipped   with its natural (Braconier)  topology (see, e.~g., \cite[(26.1)]{HiR}).    Assume that the map $u\mapsto  A(u)$ from $\Omega$ to $\mathrm{Aut}(G)$ is measurable with respect to the measure $\mu$ in $\Omega$ and the Borel structure in $\mathrm{Aut}(G)$.
  In this case all the conditions of Definition \ref{egree}      are fulfilled for $(A(u))_{u\in \Omega}$.  Indeed,  we have  $m(A(u))={\rm mod}(A(u))$,       the modulus of $A(u)$,  and the map $A\mapsto {\rm mod}(A)$ from  $\mathrm{Aut}(G)$ to $(0,\infty)$ is continuous (see \cite[(26.21)]{HiR}). Thus, the family $(A(u))_{u\in \Omega}$  agrees with the measure $\nu$ and the Proposition \ref{Lp} is applicable if $G$ is sigma-compact. 
}

\end{example}

\subsection{The general case} (See Example \ref{Kothe}.)  We accept the notion of  a rearrangement invariant  Banach 
 function space on $\frak{S}$ in the sense of \cite{BS}.

\begin{definition}\label{measure} 
Let $(\frak{S}, \mathcal{B}, \nu)$ be a   measure space. We endow the set $\mathrm{Aut}(\frak{S})$ of all automorphisms of this space  by the smallest Borel structure (sigma-algebra) $\mathfrak{A}$ such that all maps $\mathrm{Aut}(\frak{S})\to \frak{S}$, $\tau\mapsto \tau(x)$ are measurable for all  $x\in \frak{S}$. We call a family  $(\tau(u)^{-1})_{u\in \Omega}\subseteq \mathrm{Aut}(\frak{S})$ measurable, if the map $\Omega\to \mathrm{Aut}(\frak{S})$, $u\mapsto \tau(u)^{-1}$ is measurable.
\end{definition}

In the rest of the article we will use the following lemma, which has proven to be a convenient tool for solving problems of boundedness of Hausdorff-type operators in various function spaces
(in addition to the results below, see, e.g.,   \cite{KM}, \cite{MCAOT, MCAOT2}).

\begin{lemma}\label{lm1} \cite[Lemma 2]{JMAA} \textit{Let $(\frak{S};\nu)$ be a measure space,
$\mathcal{F}(\frak{S})$
 some Banach space of $\nu$-measurable functions on $\frak{S}$,  $(\Omega,\mu)$ a $\sigma$-compact quasi-metric space with  positive Radon measure $\mu$, and $F(u, x)$ a function on $\Omega\times \frak{S}$. Assume that}

(a) \textit{the convergence of a sequence in norm in $\mathcal{F}(\frak{S})$ yields the convergence of some
subsequence to the same function for $\nu$-a.~e. $x\in \frak{S}$; }

(b)  \textit{$F(u, \cdot) \in \mathcal{F}(\frak{S})$ for $\mu$-a.~e. $u\in \Omega$;}

 (c) \textit{the map  $u\mapsto F(u, \cdot):\Omega\to \mathcal{F}(\frak{S})$ is Bochner integrable with respect to } $\mu$.

  \textit{Then for $\nu$-a.~e. $x\in  \frak{S}$ one has}
$$
\left((B)\int_\Omega F(u, \cdot)d\mu(u)\right)(x)=
\int_\Omega F(u, x)d\mu(u)
$$
($(B)$ stands for the Bochner integral in $\mathcal{F}(\frak{S})$).
\end{lemma}

\begin{theorem}\label{Kote} 
Let $\Omega$
 be a $\sigma$-compact quasi-metric space with  Radon
measure $\mu$. Let $(\frak{S}, \mathcal{B}, \nu)$ be a complete $\sigma$-finite separable  measure space, and $X$ be  a rearrangement-invariant Banach 
 function space on $\frak{S}$ with an absolute continuous norm which is 
invariant under all  automorphisms of $\frak{S}$.  If a  family  $(\tau(u)^{-1})_{u\in \Omega}$ of automorphisms of $\frak{S}$ is measurable (see the  Definition \ref{measure} above), then the corresponding  operator $\mathcal{H}_{\Phi,\tau}$  in the Example \ref{Kothe} is bounded in $X$ provided $\Phi\in L^1(\mu)$ and its norm does not exceed $C_X\|\Phi\|_{L^1(\mu)}$ where the constant $C_X$ depends on $X$ only. Conversely, if this operator is bounded in $X$ and $1\in X$, then $\Phi\in L^1(\mu)$.
\end{theorem}

\begin{proof} We shall verify the conditions of Lemma \ref{lm1} where  $\mathcal{F}(\frak{S})=X$, $F(u,x)=\Phi(u)f(\tau(u)^{-1}(x))$.

(a) By \cite[Chapter I, Theorem 1.4]{BS} the convergence in $X$ implies the convergence  of some
subsequence to the same function for $\nu$-a.~e. $x\in \frak{S}$.

(b) This follows from the invariance of $X$.

(c)  Since $X$ is separable \cite[Chapter I, Corollary 5.6]{BS}, to verify  that the $X$-valued function $u\mapsto F(u,\cdot)$ is strongly $\mu$-measurable it suffices to prove that the $X$-valued function $u\mapsto f\circ \tau^{-1}(u)$  is weakly $\mu$-measurable.
To do this, note that every linear bounded functional on $X$ has the form
$$
L_g(f)=\int_{\frak{S}}f(x)g(x)d\nu(x)
$$
for some measurable function $g$ on $\frak{S}$ \cite[Chapter I,, Corollary 4.3]{BS}. Thus, if  a family  $(\tau(u)^{-1})_{u\in \Omega}$  is measurable, the function 
$$
u\mapsto L_g( f\circ \tau^{-1}(u))=\int_{\frak{S}}f\circ \tau(u)^{-1}(x)g(x)d\nu(x)
$$
is measurable, too, since for each $x\in \frak{S}$ the map $u\mapsto f\circ \tau(u)^{-1}(x)$  is measurable as a composition of measurable mappings.  
Next,
it is known  \cite{Lux} (see also \cite[p. 115] {LiTz})
  that  the   family of  operators of the form $U_\tau f=f\circ \tau(u)^{-1}$ is  uniformly bounded in  $X$, i.~e.  $\|U_\tau\|_{X\to X}\le C_X$  where the constant $C_X$ depends on $X$ only. Since $\|\Phi (u)f\circ \tau(u)^{-1}\|_X\le C_X|\Phi(u)|\|f\|_X$
  and $\Phi\in L^1(\mu)$, this implies (c). 

Thus, by  Lemma \ref{lm1},
\begin{equation*}
\mathcal{H}_{\Phi,\tau}f =\int_{\Omega} \Phi(u)f\circ \tau(u)^{-1}d\mu(u)
\end{equation*}
(the Bochner integral for $X$), and therefore
\begin{eqnarray*}
\|\mathcal{H}_{\Phi,\tau}f\|_{X}& \le&\int_{\Omega} |\Phi(u)|\|U_{\tau(u)}f\|_X d\mu(u)\\
&\le&C_X\|\Phi\|_{L^1(\mu)}\|f\|_X.
\end{eqnarray*}

The last statement of the theorem is obvious. 
\end{proof}

\section{ Boundedness of Hausdorff-type operators on  Sobolev spaces.}
\label{sec:sobolev}\setcounter{equation}{0}


In this section,  $(\frak{S}, \rho, \nu)$   stands for a connected metric measure space,  and $\rm{Aut}(\frak{S})=\rm{Iso}(\frak{S})$  the group of bijective isometries of $\frak{S}$.  Equipped  with the compact-open topology $\rm{Iso}(\frak{S})$ becomes a locally compact topological group \cite[Theorem 4.7]{KobNom}.

\begin{definition}\label{mes_Iso}
We call a family $(A(u))_{u\in \Omega}\subseteq \rm{Iso}(\frak{S})$ measurable if the map $u\mapsto A(u)$ is  measurable as a map between  the measure space $(\Omega,\mu)$ and  the topological space $\rm{Iso}(\frak{S})$.
\end{definition}

Recall  (see, e.g., \cite{Heinonen})
that a Haj{\l}asz-Sobolev space $M^{1,p}(\frak{S})$ ($1\le p<\infty$) consists of such  functions $f\in L^p(\nu)$  that there exists 
$g\in L^p(\nu)$ with the property that the inequality
\begin{align}\label{Hajlasz}
|f(x)-f(y)|\le \rho(x,y)(g(x)+g(y))
\end{align}
holds for $\nu$-a.e. $x, y\in \frak{S}$. 
This space, equipped with the following norm
$$
\|f\|_M:=\|f\|_{L^p(\nu)}+\inf \|g\|_{L^p(\nu)}
$$
where infimum is taken over all $g\in L^p(\nu)$ satisfying the defining inequality \eqref{Hajlasz} becomes a Banach space. 

\begin{definition}\label{tame}
We shall call a Haj{\l}asz-Sobolev space $M^{1,p}(\frak{S})$ tame if each linear bounded functional $L$ on $M^{1,p}(\frak{S})$ has the form
\begin{align}\label{L}
L(f)=\sum_{j=1}^\infty \int_\frak{S} D_jf(x)v_j(x)d\nu(x)
\end{align}
for some sequence $D_j$ of bounded operators on $M^{1,p}(\frak{S})$  and functions $v_j\in L^{q}(\nu)$ ($1/p+1/q=1$) such that
$$
\sum_{j=1}^\infty \|D_j\|\|v_j\|_{L^q(\nu)}<\infty.
$$
\end{definition}

Before we formulate our next theorem recall that $\|\Phi\|_{A,p}$ $=\|\Phi(\cdot)m(A(\cdot))^{-1/p}\|_{L^1(\mu)}$.

\begin{theorem}\label{Sobolev} 
Let $\Omega$
 be a $\sigma$-compact quasi-metric space with  Radon
measure $\mu$. Let  $(\frak{S}, \rho, \nu)$   stands for a connected metric measure space and a  family  $(A(u))_{u\in \Omega}$ $\subseteq \rm{Iso}(\frak{S})$ be measurable and 
agrees with the measure $\nu$. The corresponding Hausdorff-type operator $\mathcal{H}_{\Phi,A}$  is bounded in a tame and   separable Haj{\l}asz-Sobolev space $M^{1,p}(\frak{S})$  ($1\le p<\infty$) if  $\Phi(\cdot)m(A(\cdot))^{-1/p}\in L^1(\mu)$, and  in this case
$$
\|\mathcal{H}_{\Phi,A}\|\le \|\Phi\|_{A,p}.
$$ 
\end{theorem}

\begin{proof} Let $\Phi(\cdot)m(A(\cdot))^{-1/p}\in L^1(\mu)$.
 We shall verify the conditions of Lemma \ref{lm1} where  $\mathcal{F}(\frak{S})=M^{1,p}(\frak{S})$, $F(u,x)=\Phi(u)f(A(u)(x))$. 
 
 (a) Since $\|f\|_{L^p(\nu)}\le \|f\|_M$, this follows from the well known F. Riesz theorem.
 
 (b) 
 For  $f\in M^{1,p}(\frak{S})$ and $A(u)$ ($u\in \Omega$) we have for a.~e. $x, y\in \frak{S}$
\begin{align*}
|f(A(u)(x))-f(A(u)(y))|\\
\le &\rho(A(u)(x),A(u)(y))(g(A(u)(x))+g(A(u)(y)))\\ \nonumber
=&\rho(x,y)(g_u(x)+g_u(y)),  \nonumber
\end{align*} 
where $g_u:=g\circ A(u)\in  L^p(\nu)$ by \eqref{IfA(u)}. Thus, $f\circ A(u)\in  M^{1,p}(\frak{S})$ for all $u\in \Omega$ and (b) holds.

(c)  We shall use the criterium of Bochner integfrability ones more  (see, e.g., \cite{Hyt}). 
Since the  space $M^{1,p}(\frak{S})$ is separable, to verify  that the $M^{1,p}(\frak{S})$-valued function $u\mapsto F(u,\cdot)$ is strongly $\mu$-measurable it suffices to prove that the $M^{1,p}(\frak{S})$-valued function $u\mapsto f\circ A(u)$  is weakly $\mu$-measurable.
To do this, note that, as  a family  $(A(u))_{u\in \Omega}$  is measurable and all operators $D_j$ in \eqref{L} are  continuous, the function 
$$
u\mapsto L( f\circ A(u))=\sum_{j=1}^\infty\int_{\frak{S}}D_jf(A(u)(x))v_j(x)d\nu(x)
$$
is measurable, too for each linear bounded functional $L$ on $M^{1,p}(\frak{S})$ of the form \eqref{L}. Indeed,  for every  $x\in \frak{S}$ the  map $A\mapsto A(x)$, $\rm{Iso}(\frak{S})\to \frak{S}$ is continuous, and thus the map $u\mapsto f\circ A(u)(x)$  is measurable as a composition of measurable mappings. 
 
Next, for each  $u\in \Omega$ and every $f\in M^{1,p}(\frak{S})$
\begin{align*}
\|f\circ A(u)\|_M&=\|f\circ A(u)\|_{L^p(\nu)}+\inf \|g_1\|_{L^p(\nu)}\\
&=m(A(u))^{-1/p}\|f\|_{L^p(\nu)}+\inf \|g_1\|_{L^p(\nu)}
\end{align*}
where infimum is taken over all $g_1\in L^p(\nu)$ satisfying the condition
$$
|f(A(u)(x))-f(A(u)(y))|\le\rho(x,y)(g_1(x)+g_1(y))
$$
for a.e. $x,y\in \frak{S}$. Since every such function has the form $g_1=g\circ A(u)$ where $g=g_1\circ A(u)^{-1}$ satisfies the condition \eqref{Hajlasz}, and $\|g_1\|_{L^p(\nu)}=m(A(u))^{-1/p}\|g\|_{L^p(\nu)}$ by \eqref{IfA(u)}, we have $\|f\circ A(u)\|_M=m(A(u))^{-1/p}\|f\|_M$. This implies (c) due to the criterium of Bochner integfrability,  because $\Phi(\cdot)|m(A(\cdot))^{-1/p}\in L^1(\mu)$ and  
$$
\|\Phi(u)f\circ A(u)\|_M=|\Phi(u)|m(A(u))^{-1/p}\|f\|_M
$$
 for all $u\in\Omega$.

Thus, by  Lemma \ref{lm1}, 
\begin{equation*}
\mathcal{H}_{\Phi,A}f =\int_{\Omega} \Phi(u)f\circ A(u)d\mu(u)
\end{equation*}
(the Bochner integral for $M^{1,p}(\frak{S})$). 
Therefore
\begin{eqnarray*}
\|\mathcal{H}_{\Phi,A}f\|_{M}& \le&\int_{\Omega} |\Phi(u)|\|\|f\circ A(u)\|_Md\mu(u)\\
&\le&\|\Phi\|_{A,p}\|f\|_M.
\end{eqnarray*}

\end{proof}

\begin{corollary}\label{Rn} (cf. \cite{MirSob}) Let  $\frak{S}$ be a  domain in $\mathbb{R}^n$ with a smooth boundary endowed with the Euclidean metric $\rho$  and Lebesgue measure $\nu$,  and a family $(A(u))_{u\in \Omega}$ $\subseteq \rm{Iso}(\frak{S})$ be measurable. If $\Phi\in L^1(\mu)$ then $\mathcal{H}_{\Phi,A}$ is bounded in a classical Sobolev space $W^{1,p}(\frak{S})$ for $1<p<\infty$.
\end{corollary}

\begin{proof}
 Indeed, in this case $W^{1,p}(\frak{S})$ is isomorphic to $M^{1,p}(\frak{S})$ with an equivalent norm (see \cite[p. 40]{Heinonen}), and thus this space is   separable  (see, e.~g., \cite[Theorem 3.6]{Adams}), and tame (see, e.~g., \cite[Theorem 3.19]{Adams}). Moreover,  each isometry $A\in \rm{Iso}(\frak{S})$ preserves the Lebesgue measure $\nu$ in $\frak{S}$.   This follows from the main result in  \cite{Zeleny} (see  p. 433  therein). Therefore  $(A(u))_{u\in \Omega}$
agrees with the measure $\nu$ with $m(A(u))\equiv 1$. Thus, all the conditions of Theorem \ref{Sobolev} are satisfied for $W^{1,p}(\frak{S}) \simeq M^{1,p}(\frak{S})$ and $\Phi$.
\end{proof}

\section{$H^1$  boundedness of  Hausdorff-type operators}
\label{H10}\setcounter{equation}{0}

\subsection{The  general case}\setcounter{equation}{0}
In this subsection we shall be working in the following setting. 
 We  assume that  $\frak{S}$ is  a   quasi-metric space with
  quasi-metric $\rho$ and positive regular Borel measure $\nu$. Moreover, the following \textit{doubling condition} holds:

  There exists of a constant $C$ such that
$$
\nu(B(x,2r))\leq C \nu(B(x,r))
$$
for each $x\in \frak{S}$ and $r > 0$. 

 (Here  and below $B(x,r)=B_\rho(x,r)$  denotes a quasi-ball  with respect to $\rho$ with a center $x$ and radius $r>0$).

In this case, the triple  $(\frak{S},\rho,\nu)$ is called  \textit{a   quasi-metric measure  space of homogeneous type} \cite{CW}.

 The \textit{doubling constant }  is the smallest constant
$C\geq  1$ for which the doubling inequality holds. We denote this constant by $C_\nu.$ Then  for each $x\in  \frak{S}, k\geq 1$ and $r > 0$
\begin{equation}\label{D}
\nu(B(x,kr))\leq C_\nu k^{s} \nu(B(x,r)),
\end{equation}
where $s=\log_2C_\nu$ (see, e.g., \cite[p. 76]{HK}). The number $s$  sometimes takes
the role of a ``dimension'' for a doubling quasi-metric measure space.

\begin{definition}\label{*}
Let $(\Omega,\mu)$ be a measure space. We say that a family of  automorphisms $(A(u))_{u\in \Omega}$  of a quasi-metric space $(\frak{S},\rho)$ \textit{agrees with the quasi-metric $\rho$}  if  there exists a $\mu$-measurable function $k(u)$ which depends on $u\in \Omega$ only, such that  for every $x\in \frak{S}$,  for every $u\in \Omega$, and for every  $r>0$
\begin{equation}\label{ast}
A(u)^{-1}(B(x,r))\subseteq B(x',k(u)r) 
\end{equation}
for some  point $x'=x'(x,u,r)\in \frak{S}$\footnote{In fact, $k(u)$ depends on $A(u)$.}.
\end{definition}

\begin{remark}\label{rem1}
{\rm  Let $\Omega$
 be a $\sigma$-compact quasi-metric space with  Radon
measure $\mu$. If $\frak{S}=G$ is a (finite dimensional real or complex) connected Lie group
with (left or right) invariant Riemann metric $\rho$, then every automorphism $A\in \mathrm{Aut}(G)$ is
Lipschitz and  $\mathrm{Aut}(G)$ agrees with $\rho$ by \cite[Lemma 2.6]{Lie}.
}
\end{remark}

\begin{definition}\label{munu} 
Let $(\Omega,\mu)$ be a measure space.
We say that a family of  automorphisms $(A(u))_{u\in \Omega}$ of $\frak{S}$ is $\mu$-$\nu$  measurable if  for every $x\in \frak{S}$
the map $u\mapsto A(u)(x)$ from $(\Omega,\mu)$ to $(\frak{S},\nu)$ is measurable.
\end{definition}

Recall \cite{CW} that a $\nu$-measurable function $a$ on $\frak{S}$ is an $(1,q)$-\textit{atom} ($q\in (1,\infty]$) if

(i) the support of $a$ is contained in a ball $B(x, r)$;

(ii) $\|a\|_\infty \le \frac{1}{\nu(B(x,r))}$ if $q=\infty$, and

$\|a\|_q \le \nu(B(x,r))^{\frac{1}{q}-1}$  if $q\in (1,\infty)$\footnote{ $\|\cdot\|_q$ denotes the $L^q$ norm.};

(iii) $\int_G a(x)d\nu(x) = 0$.

In case $\nu(\frak{S})<\infty$ we shall assume $\nu(\frak{S})=1$; in this case  the constant function
having value $1$ is also considered to be an atom.

From now on  by atom we mean an $(1,q)$-atom.

\begin{definition} \cite[p. 592]{CW}  Let $q\in (1,\infty]$. We define the \textit{Hardy space} $H^{1,q}(\frak{S})$ as a space
of such functions $f $ on $\frak{S}$ that $f$ admits an atomic decomposition of
the form
\begin{equation}\label{3}
f =\sum_{j=1}^\infty \alpha_ja_j,  
\end{equation}
where $a_j$ are   $(1,q)$-atoms,
and $\sum_{j=1}^\infty |\alpha_j| < \infty$ (the sums \eqref{3} are convergent in the $L^1$ norm). \footnote{It is known that $H^{1,q}$ does not depend on $q\in (1,\infty]$ \cite[Theorem A, p. 592]{CW}. We write $H^{1,q}$ instead of $H^{1}$ in order to stress the fact that we use the  norm $\|\cdot\|_{H^{1,q}}$ described below.}
 In this case,
$$
\|f\|_{H^{1,q}(\frak{S})} := \inf \sum_{j=1}^\infty |\alpha_j|,
$$
and infimum is taken over all decompositions above of $f$.
\end{definition}

Since $\|a\|_{L^1}\le 1$, one has  $\|f\|_{L^{1}(\frak{S})}\le \|f\|_{H^{1,q}(\frak{S})}$  for a function $f$ in $H^{1,q}(\frak{S})$, in particular $H^{1,q}(\frak{S})\subset L^{1}(\frak{S})$.

Let $X$ be a Banach space and $(\Omega,\mu)$ be a measure space. Recall that a function $g:\Omega\to X$ is said to be $\mu$-essentially separably valued if there
exists a closed separable subspace $X_0$ of $X$ such that $g(u) \in X_0$ for $\mu$-almost
all $u\in \Omega$, and weakly $\mu$-measurable if $l^*\circ g$
 is $\mu$-measurable for every bounded linear functional $l^*$ on $X$.

\begin{definition}\label{septype} Let $(\Omega,\mu)$ be a measure space. The family of  automorphisms $(A(u))_{u\in \Omega}$ of $\frak{S}$ is said to be of separable type if for each  $f\in H^{1,q}(\frak{S})$
such that $\forall u \ f\circ A(u)\in H^{1,q}(\frak{S})$ \footnote{It will be shown in the proof of the following theorem that   $f\circ A(u)\in H^{1,q}(\frak{S})$  $\forall u \in \Omega$ and $\forall f\in H^{1,q}(\frak{S})$.} the function $u\mapsto f\circ A(u)$  is  $\mu$-essentially separably valued.
\end{definition}

\begin{remark}\label{rem2} {\rm Let $(\Omega,\mu)$ be a measure space. A family of  automorphisms $(A(u))_{u\in \Omega}$ of $\frak{S}$ is  of separable type, e.g.,  in the following tree cases:

1) the space $H^{1,q}(\frak{S})$ is separable;

2)  $\Omega$ is countable;

3)  $\Omega$ is a separable  metric space, $\mu$ is a Radon measure,  and the map  $u\mapsto f\circ A(u)$, $\Omega\to H^{1,q}(\frak{S})$ is measurable $\forall f\in H^{1,q}(\frak{S})$(this follows from \cite[Lemma 1.1.12]{Hyt}).
}
\end{remark}

If a family $(A(u))_{u\in \Omega}$ of automorphisms of $\frak{S}$
agrees with the measure $\nu$ then we put
$$
N(\Phi,A,q)=C_\nu^{1-\frac 1q}\int_{\Omega}|\Phi(u)|k(u)^{s\left(1-\frac 1q\right)}m(A(u))^{-\frac 1q}d\mu(u).
$$

\begin{theorem}\label{H1}
Let $\Omega$
 be a $\sigma$-compact quasi-metric space with positive Radon
measure $\mu$ and  let a set $\frak{S}$ be an object of some category $\mathfrak{K}$ and  in addition    $(\frak{S},\rho,\nu)$ be a   quasi-metric measure  space of homogeneous type.
 If a  $\mu$-$\nu$- measurable family of  automorphisms $(A(u))_{u\in \Omega}$ of $\frak{S}$ (in the category $\mathfrak{K}$)  is of separable type and agrees with the quasi-metric $\rho$  and with the  measure $\nu$, and 
$N(\Phi,A,q)<\infty $,
then a Hausdorff operator  $\mathcal{H}_{\Phi, A}$ is bounded in $H^{1,q}(\frak{S})$ ($q\in(1,\infty]$) and its norm does not exceed  $N(\Phi,A,q)$. 
\end{theorem}

\begin{proof}
We use the approach from \cite{JMAA}. First we are going to show that   the conditions of Lemma \ref{lm1} are fulfilled  with $\mathcal{F}(\frak{S})=H^{1,q}(\frak{S})$ and $F(u,x)=\Phi(u)f(A(u)(x))$ where $f\in H^{1,q}(\frak{S})$.

Let $1<q<\infty$. Since  for a function $f\in H^{1,q}(\frak{S})$ one has  $\|f\|_{L^{1}(\frak{S})}\le \|f\|_{H^{1,q}(\frak{S})}$, the condition (a) of the Lemma follows from the well known theorem of F.~ Riesz.

To verify conditions (b) and (c), consider  a function $f\in H^{1,q}(\frak{S})$ with an atomic representation \eqref{3}. Then
\begin{equation}\label{4}
f\circ A(u) =\sum_{j=1}^\infty \alpha_ja_j\circ A(u) ,  
\end{equation}
for all $u\in\Omega$. 
We claim that a function
$$
a'_{j,u}:=C_\nu^{\frac 1q-1}k(u)^{s(\frac 1q-1)}m(A(u))^{\frac 1q}a_j\circ A(u)
$$
is an atom, as well. 

Indeed, if an atom  $a_j$ is supported in a ball $B(x_j,r_j)$ then $a'_{j,u}$  is supported in  $A(u)^{-1}(B(x_j,r_j))\subseteq B(x'_j,k(u)r_j)$ by \eqref{ast}.

Next, since the property (ii) holds for $a_j$ , we have by \eqref{IfA(u)}
\begin{eqnarray}\label{a'}
\|a'_{j,u}\|_q&=&C_\nu^{\frac 1q-1}k(u)^{s(\frac 1q-1)}m(A(u))^{\frac 1q}\|a_j\circ A(u)\|_q\\ \nonumber
&=&C_\nu^{\frac 1q-1}k(u)^{s(\frac 1q-1)}m(A(u))^{\frac 1q}\left(\int_{\frak{S}} |a_j(A(u)(x))|^qd\nu(x)\right)^{\frac 1q}\\ \nonumber
&=&C_\nu^{\frac 1q-1}k(u)^{s(\frac 1q-1)}\|a_j\|_q\\ \nonumber
&\le&C_\nu^{\frac 1q-1}k(u)^{s(\frac 1q-1)}\nu(B(x_j,r_j))^{\frac 1q-1}.
\end{eqnarray}

On the other hand, the doubling condition \eqref{D} yields
$$
\nu(B(x'_j,k(u)r_j))\le C_\nu k(u)^s\nu(B(x_j,r_j))
$$
and therefore
$$
\nu(B(x_j,r_j))^{\frac 1q-1}\le  (C_\nu  k(u)^s)^{1-\frac 1q}\nu(B(x'_j,k(u)r_j))^{\frac 1q-1}.
$$

Now \eqref{a'} implies that
$$
\|a'_{j,u}\|_q\le \nu(B(x'_j,k(u)r_j))^{\frac 1q-1},
$$
i.e., (ii) holds for $a'_{j,u}$.

Finally, the cancellation condition (iii)  for $a'_{j,u}$ follows from  \eqref{IfA(u)} and the corresponding condition  for $a_{j}$.

Further, since for all $u\in\Omega$
$$
a_j\circ A(u)=C_\nu^{1-\frac 1q}k(u)^{s(1-\frac 1q)}m(A(u))^{-\frac 1q}a'_{j,u},
$$
formula \eqref{4} reeds as
$$
f\circ A(u) =\sum_{j=1}^\infty \left(\alpha_j C_\nu^{1-\frac 1q}k(u)^{s(1-\frac 1q)}m(A(u))^{-\frac 1q}\right)a'_{j,u}.
$$
It follows that $f\circ A(u)\in H^{1,q}(\frak{S})$ (and therefore the condition  (b) holds) and
\begin{equation}\label{c}
\|f\circ A(u)\|_{ H^{1,q}}\le \left(C_\nu^{1-\frac 1q}k(u)^{s(1-\frac 1q)}m(A(u))^{-\frac 1q}\right)\|f\|_{ H^{1,q}}.
\end{equation}

The condition  (c) holds, too. Indeed, since the family $(A(u))$ is of separable type, to verify  that the $H^{1,q}(\frak{S})$-valued function $u\mapsto f\circ A(u)$ is strongly $\mu$-measurable it suffices to prove that it is weakly $\mu$-measurable. To this end, in view of \eqref{3}, it suffices to consider the case where $f=a$ is an atom. Let $l^*$ be a linear continuous functional on $H^{1,q}(\frak{S})$. Then \cite[Theorem B]{CW} there is such a function $l\in BMO(\frak{S})$ that
$$
l^*( a\circ A(u))=\int_{\frak{S}}l(x) a(A(u)(x))d\nu(x).
$$
The map $u\mapsto l^*( a\circ A(u))$ is $\mu$-measurable, if the map $\phi(u):= a\circ A(u)(x)$ is $\mu$-measurable for each $x$. To verify the last property one can assume that $a$ is real-valued. Let $E_c=\{y\in \frak{S}: a(y)<c\}$ \ ($c\in \mathbb{R}$). Then $E_c$ is $\nu$-measurable and so  the set $\phi^{-1}((-\infty,c))=\{u\in\Omega: A(u)(x)\in E_c\}$
 is $\mu$-measurable by Definition \ref{munu}.

Now the inequality \eqref{c} and the condition $N(\Phi,A,q)<\infty$ imply that the function $u\mapsto \|F(u,\cdot)\||_{ H^{1,q}}$ is  Lebesgue $\mu$-integrable and (c) from  the Lemma \ref{lm1} holds. 

Thus, by  Lemma \ref{lm1},
\begin{equation*}
\mathcal{H}_{\Phi,A}f =\int_{\Omega} \Phi(u)f\circ A(u)d\mu(u)
\end{equation*}
(the Bochner integral), and therefore
\begin{eqnarray*}
\|\mathcal{H}_{\Phi,A}f\|_{ H^{1,q}}& \le&\int_{\Omega} |\Phi(u)|\|f\circ A(u)\|_{ H^{1,q}}d\mu(u)\\
&\le&N(\Phi,A,q)\|f\|_{ H^{1,q}}.
\end{eqnarray*}
The case $q=\infty$ can be treated in a similar manner.

 The proof is complete.
\end{proof}

The condition $N(\Phi, A,q)<\infty$ is not necessary for boundedness of
$\mathcal{H}_{\Phi, A}$ in $H^{1,q}$, see remark \ref{nonnes} below.

\begin{example}\label{Hankel2} {\rm The real Hardy space $H^{1,q}(\mathbb{R}_+)$ can be considered as  a subspace of $H^{1,q}(\mathbb{R})$ (we extend every function from $H^{1,q}(\mathbb{R}_+)$ to $(-\infty,0)$ as zero). Denote this  natural embedding by $J$. If $\varphi\in L^{1}(\mathbb{R})$, then  the operator $H_\varphi$ from  Example \ref{Hankel} is bounded in $H^{1,q}(\mathbb{R})$  by  Theorem \ref{H1}. It follows that an integral Hankel operator $\Gamma_\varphi=H_\varphi J$ is a bounded operator between $H^{1,q}(\mathbb{R}_+)$ and  $H^{1,q}(\mathbb{R})$ and its norm does not exceed  $2^{1-\frac{1}{q}}\|\varphi\|_{L^1}$. (Indeed, in our case $\Omega=\frak{S}=\mathbb{R}$, $C_\nu =2$, and $k(u)\equiv 1$, $m(A(u))\equiv 1$.)
}
\end{example}

\subsection{The group case}\label{group}\setcounter{equation}{0}

Below we shall  assume that $G$ is a    locally compact  group and the  group $\mathrm{Aut}(G)$  of all topological automorphisms of  $G$ is equipped   with its natural (Braconnier) topology. In this topology the sets 
$$
\mathcal{O}(C, V):=\{A\in \mathrm{Aut}(G): A(x)x^{-1}\in V,   A^{-1}(x)x^{-1}\in V \forall x\in C\} 
$$
where     $C$ runs over  all
compact subsets of $G$ and $V$ runs over all neighborhoods of the unit  $e\in G$
constitute a fundamental system of neighborhoods of the identity  (see, e.~g., \cite[(26.1)]{HiR}, \cite[Section III.3]{Hoch}).

\begin{definition}\label{measAut} Let $(\Omega,\mu)$ be a $\sigma$-compact quasi-metric space
with   positive Radon measure $\mu$. A  family of topological automorphisms $(A(u))_{u\in \Omega}$ of a locally compact  group  $G$  is called measurable  if  the map $u\mapsto  A(u)$ is measurable with respect to the measure $\mu$ and the Borel structure in the topological space $\mathrm{Aut}(G)$
\footnote{Here we specify  the notion of the measurable family of topological automorphisms  of a locally compact  group from \cite{JMAA, Lie, homogen, JOTH, Nachr}.}.
\end{definition}

\begin{theorem}\label{group2} (cf. \cite{JMAA}, \cite{Lie}).  Let $(\Omega,\mu)$ be a $\sigma$-compact quasi-metric space
with   positive Radon measure $\mu$. Let $\frak{S}=G$ be a locally compact  group with
 Haar measure $\nu$, and the topology of $G$ is generated by a quasi-metric $\rho$.  Assume that    $(G,\rho,\nu)$ is a    space of homogeneous type.
 If a measurable family of topological automorphisms $(A(u))_{u\in \Omega}$ of  $G$   is of separable type, agrees with the quasi-metric $\rho$,  and $N(\Phi,A,q)<\infty $,
then a Hausdorff operator  $\mathcal{H}_{\Phi, A}$ is bounded in $H^{1,q}(G)$ and its norm does not exceed  $N(\Phi,A,q)$. 
\end{theorem}

\begin{proof} The only conditions of Theorem \ref{H1} we need to verify are that the family  $(A(u))_{u\in \Omega}$ 
agrees with the   measure $\nu$ and that it is   $\mu$-$\nu$-measurable. 

For the proof of the first property note that in  our case we have $m(A(u))=\mathrm{mod}(A(u))$, and the map $u\mapsto \mathrm{mod}(A(u))$ is $\mu$-measurable, since the family $(A(u))_{u\in \Omega}$ is measurable, and the map $A\mapsto \mathrm{mod}A$ from  $\mathrm{Aut}(G)$ to $(0,\infty)$ is continuous (see \cite[(26.21)]{HiR}).

Finally, since for each $x\in G$ the map  $\alpha_{x}:\mathrm{Aut}(G) \to G$ sending $A$ onto $A(x)$ is continuous \cite[Proposition III.3.1, p. 40]{Hoch}, and  the family $(A(u))_{u\in \Omega}$ is measurable, it is   $\mu$-$\nu$-measurable.
\end{proof}

\begin{remark}\label{rem3} {\rm Since $\Omega$ in Theorem \ref{group2}  is  separable, and the map $u\mapsto  A(u)$ is measurable,  Lemma 1.1.12 from \cite{Hyt}  shows that $(A(u))_{u\in \Omega}$  is of  separable type if the map $A\mapsto f\circ A$,   $\mathrm{Aut}(G) \to H^{1,q}(G)$ is measurable (see the  Conjecture below).
}
\end{remark}

{\bf Conjecture.} Let $G$ be a locally compact  group with
 Haar measure $\nu$, and the topology of $G$ is generated by a quasi-metric $\rho$. If $(G,\rho,\nu)$ is a     space of homogeneous type, then the map $A\mapsto f\circ A$,   $\mathrm{Aut}(G) \to H^{1,q}(G)$ is measurable for every $f\in H^{1,q}(G)$.

\begin{example}\label{delsarte1} {\rm Consider the generalized   shift operator of Delsarte (see Example \ref{delsarte}) over a connected Lie group $G$ with  left invariant Haar measure $\nu$, and  left invariant Riemann metric $\rho$
     $$
  T^{x}f(h):=\int_\Omega f(h u(x)))d\mu(u)\quad (x\in G, h\in G).
  $$
  Let $h$ be fixed and $T_hf(x):=T^{x}f(h)$.
Then $T_h=\mathcal{H}_1S_h$, where
$$
\mathcal{H}_1f( x):=\int_\Omega f(u(x))d\mu(u)
$$
is a Hausdorff operator on $G$ with $\Phi(u)\equiv 1$, $A(u)=u$, and $S_h f(x):=f(h x)$.  Note that  $\mathrm{mod}$ is a continuous
homomorphism from $\mathrm{Aut}(G)$ to the multiplicative group $(0,\infty)$ ($\mathcal{H}_1$ is a Ces\'aro operator in the sense of \cite{JMAA}). It follows that  $\mathrm{mod}(\Omega)=\{1\}$. This  operator is bounded on $L^p(G)$ ($1\le p\le\infty$) by Proposition \ref{Lp}.

 The family $\mathrm{Aut}(G)$ agrees with $\rho$ and $k(u)=\|(du^{-1})_e\|$ by \cite[Lemma 2.6]{Lie}.   Assume that the group $(G,\rho,\nu)$ is doubling. Since $\mathrm{Aut}(G)$ is a Lie group, 
 the (compact) group  $\Omega\subset \mathrm{Aut}(G)$  is  metric.  Since $k(u)$ is continuous \cite[Chapter III, \S 10, Theorem 1]{bourblie},
$$
N(1,\Omega, q)=C_\nu^{1-\frac 1q}\int_\Omega k(u)^{s(1-\frac 1q)}d\mu(u)<\infty.
$$
Assume that the space  $H^{1,q}(G)$ is separable. Then operator $\mathcal{H}_1$
is bounded on $H^{1,q}(G)$  by Theorem \ref{group2} and $\|\mathcal{H}_{1}\|\le N(1,\Omega, q)$. Taking into account that $S_h$ is an isometry of $H^{1,q}(G)$, we conclude that
the operator $T_h$ is bounded on $H^{1,q}(G)$, too,  and $\|T_h\|\le N(1,\Omega, q)$.
}
\end{example}

\begin{remark}\label{nonnes} \cite{Lie} {\rm  The following  example shows that the condition $N(\Phi, A,q)<\infty$ is not necessary for boundedness of
$\mathcal{H}_{\Phi, A}$ in $H^{1,q}$.
 Consider the Hausdorff operator
$$
(\mathcal{H}_0f)(x):=\int_{\Omega}f(u_1x_1,\dots,u_nx_n)du
$$
in $ H^{1,q}(\mathbb{R}^n)$.  Here $G=\mathbb{R}^n$, $\Omega=\{u\in\mathbb{R}^n: \min_j|u_j|=1\}$, $\mu$ and $\nu$ are  Lebesgue measures on $\Omega$   and $\mathbb{R}^n$ respectively,  $A(u)(x)=A_ux$, where $A_u=\mathrm{diag}\{u_1,\dots,u_n\}$ ($x\in \mathbb{R}^n$ a column vector, $u\in \Omega$), $\Phi\equiv 1$. The necessary moment condition  $\int_{\mathbb{R}^n}f(y)dy=0$ for functions in $H^{1,q}(\mathbb{R}^n)$  yields that $\mathcal{H}_0f=0$ for all $f\in H^{1,q}(\mathbb{R}^n)$. On the other hand, here $\mathrm{mod}A(u)=|\det A_u|=|u_1\dots u_n|$ \cite[Subsection VII.1.10, Corollary 1]{Bourb}.
Further,   taking $x=0$ in \eqref{ast} we deduce that $k(u)\ge \max_j|u_j^{-1}|=1$. Then
\begin{align*}
N(\Phi, A,q)&=\int_{\Omega}(\mathrm{mod}A(u))^{-\frac{1}{q}}k(u)^{(1-\frac{1}{q})n}du\\
&\ge\int_{\Omega}\frac{du}{|u_1\dots u_n|}=\infty.
 \end{align*}
 }
\end{remark}

\subsection{The case of homogeneous spaces of Lie groups}\label{homogen}

In this subsection, $G$ denotes a Lie group with right  invariant distance  $d$ which agrees with the topology of  $G$ and left Haar  measure $\lambda_G$.   Let $K$ be a compact subgroup of  $G$ with normalized Haar measure $\lambda_K$.  Consider the quotient space
$G/K$ of left cosets $\dot x:=xK=\pi_K(x)$ ($x\in G$) where $\pi_K:G\to G/K$ stands for a natural projection. 

We equip  $G/K$ with its natural metric
$$
\rho(\dot x,\dot y)=\min_{k,k'\in K}d(xk,yk')
$$
that induces the topology of $G/K$ \cite[Theorem 1.23]{MoZi}.

We shall assume that the left-$G$-invariant measure    $\nu$ on $G/K$ is normalized in such a way that (generalized) Weil's formula
\begin{equation}\label{weil}
\int_G g(x)d\lambda_G(x)=\int_{G/K}\left(\int_Kg(xk)d\lambda_K(k)\right)d\nu(\dot x)
\end{equation}
holds for all $g\in L^1(G)$  (see
\cite[Chapter VII, \S  2, no. 5, Theorem 2 ]{Bourb} and especially  remark c) after this theorem).

We need to explain what  an automorphism of a homogeneous space means.

Consider the closed subgroup
$$
 \mathrm{Aut}_K(G):=\{A\in \mathrm{Aut}(G), A(K)=K\}
 $$
of  $\mathrm{Aut}(G)$, and  the subset
\begin{eqnarray*}
\mathrm{Aut}_K^0(G)&=&\{A\in \mathrm{Aut}_K(G): A(xK)=xK \forall x\in G\}\\
&=&\{A\in \mathrm{Aut}_K(G): x^{-1}A(x)\in K \forall x\in G\}
\end{eqnarray*}
of $ \mathrm{Aut}_K(G)$.

\begin{lemma}\label{norm:closed}
The set $\mathrm{Aut}_K^0(G)$ is a closed normal subgroup of  $\mathrm{Aut}_K(G)$.
\end{lemma}

\begin{proof} It is easy to verify that  $\mathrm{Aut}_K^0(G)$ is a subgroup of  $\mathrm{Aut}_K(G)$.

The set $\mathrm{Aut}_K^0(G)$ is closed, since the evaluation  map  $\alpha_{x}:\mathrm{Aut}(G) \to G$,  $\alpha_{x}(A):=A(x)$ is continuous \cite[Proposition III.3.1, p. 40]{Hoch}. Therefore if the net $A_n\to A$ in  $\mathrm{Aut}(G)$ and  $A_n\in \mathrm{Aut}_K^0(G)$, then $x^{-1}A(x)\in K  \forall x\in G$ and thus  $A\in \mathrm{Aut}_K^0(G)$, as well.

Finally, $\mathrm{Aut}_K^0(G)$ is an invariant subgroup of $\mathrm{Aut}_K(G)$. Indeed, for every $A\in \mathrm{Aut}_K^0(G)$ and every  $A_1\in \mathrm{Aut}_K(G)$
we have
$$
A_1^{-1}AA_1(xK)=A_1^{-1}A(A_1(x)K)=A_1^{-1}A_1(x)K=xK
$$
for all $x\in G$. So, $A_1^{-1}AA_1\in \mathrm{Aut}_K^0(G)$.
\end{proof} 

\begin{definition}\label{aut}
We define the group   $ \mathrm{Aut}(G/K)$ of automorphisms of  a  homogeneous space $G/K$  as a factor-group $\mathrm{Aut}_K(G)/\mathrm{Aut}_K^0(G)$.
\end{definition}

Let an automorphism $A\in \mathrm{Aut}_K(G)$. Since
$$
A(\dot x):= A(xK)=\{A(x)A(k): k\in K\}=
A(x)K=\pi_K(A(x)),
$$
 we get a homeomorphism   $\dot A:G/K\to G/K,$  $\dot A(\dot x):=\pi_K(A(x))$ (see, e.g., \cite[Chapter I, \S 3, Propo. 9]{gentop12}). Also $\dot A(o)=o$ where $o=\dot e$.

\begin{lemma}\label{eval1}
$\mathrm{Aut}(G/K)=\{\dot A: A\in \mathrm{Aut}_K(G)\}.$

2. For each $x\in G$  the evaluation map 
$$
\varepsilon_{\dot x}: \mathrm{Aut}(G/K)\to G/K, \varepsilon_{\dot x}(\dot A)= \dot A(\dot x)
$$
is continuous.
\end{lemma}

\begin{proof}  1. Indeed, 
$$
\mathrm{Aut}_K^0(G)=\{A_0\in \mathrm{Aut}_K(G):  A_0(\dot x)=\dot x \forall x\in G\}.
$$
Then $AA_0(\dot x)=A(\dot x)=\dot A(\dot x)$ for all $x\in G$, $A_0\in \mathrm{Aut}_K^0(G)$, and $A\in \mathrm{Aut}_K(G)$.
But  an arbitrary element of $\mathrm{Aut}_K(G)/\mathrm{Aut}_K^0(G)$  has the form $A\mathrm{Aut}_K^0(G)=\{AA_0: A_0\in \mathrm{Aut}_K^0(G)\}$  where $A\in \mathrm{Aut}_K(G)$. 
 It follows that we can identify the groups under consideration.

2. Since the map  $\alpha_{x}:\mathrm{Aut}(G) \to G$ sending $A$ onto $A(x)$ is continuous \cite[Proposition III.3.1, p. 40]{Hoch},  the map  $\beta_{x}:\mathrm{Aut}_K(G) \to G/K$ sending $A$ onto $\pi_K A(x)=(A(x))^{{\bf\cdot}}$ is continuous, too. Let $\phi$ denotes the  natural projection $\mathrm{Aut}_K(G)\to \mathrm{Aut}(G/K)$, $\phi(A)=\dot A$. Then
 $\beta_{x}=\varepsilon_{\dot x} \phi$ and thus the map $\varepsilon_{\dot x} $  is continuous (see, e.g., \cite[Chapter 1, \S 3]{gentop12}).
\end{proof}

\begin{example} (Hausdorff operators over the hyperbolic plane \cite{homogen}.) {\rm As is well known, one can  identify the hyperbolic plane $\mathbb{H}^2$ with the homogeneous space ${\rm SL}(2)/{\rm  SO}(2)$. In this case a Hausdorff-type operator looks as follows:
$$
(\mathcal{H}_{\Phi}f)(z) 
=\int_{0}^{2\pi} \Phi(\theta)f(R^\theta(z))d\mu(\theta),\quad z\in \mathbb{H}^2,
$$
where the M\"{o}bius transformation
$$
R^\theta(z):=\frac{z\cos\theta+\sin\theta}{-z\sin\theta+\cos\theta}
$$
 induces a hyperbolic rotation of the half-plane $\mathbb{H}^2$ by the angle $2\theta$ about $\imath$ (see, e.~g., \cite[Lemma 9.19]{St}).
}
\end{example}

For other examples of Hausdorff-type operators on  homogeneous spaces of groups see \cite{Lie}, \cite{homogen}.

Now we are going to apply Theorem \ref{H1}  to the space  $H^{1,q}(G/K)$.

\begin{lemma}\label{nontriv}
  If $G\ne K$ then  the space  $H^{1,q}(G/K)$ is nontrivial.
\end{lemma}

\begin{proof}  We prove that   nontrivial atoms exist. Fix a point $x_0\in G\setminus K$. Then $x_0K\cap K=\varnothing$. Since the compact set  $x_0K\cup K$ is contained in some compact ball $B_d$,  the function
$$
h(x):=
\begin{cases}
1 \ \mbox{ for } x\in K,\\
-1 \mbox{ for } x\in x_0K,\\
0  \ \mbox{ for } x\in G\setminus(x_0K\cup K)
\end{cases}
$$
is supported by  $B_d$ and $\int_Gh(x)d\lambda_G(x)=0$. Consider the function
$$
a(x):=c\int_K h(xk)d\lambda_K(k)
$$
where $c$ is a positive  constant.  Then $a$ is right-$K$-invariant in a sense that $a(xk')=a(x)$ for all $x\in G, k'\in K$. Note that, since  $h$ is supported in  $B_d$,  then $a(x)=0$ for $x\notin B_dK$. Since $B_d$ is compact,
 the  set $B_dK$ is compact, as well, and therefore it is  contained in some  ball $B_d'=B_d(e,r)$.  Thus, $a$  is supported by $B_d'$.  It follows that the function
$$
\dot a(\dot x):=a(x)
$$
on $G/K$ is supported by a ball $B_\rho(\dot e,r)$.  Indeed, 
we have 
$$
\rho(\dot x,\dot y)=\min_{k,k'\in K}d(xk,yk')=\min_{k\in K}d(xk,y)
$$ 
by the right invariance of $d$. Then $\rho(\dot x,\dot x_0)<r$ $\Longleftrightarrow$ $d(xk,x_0)<r$
for some $k\in K$. In other words, for every $x_0\in G$, $r>0$ one has $\dot x\in B_\rho(\dot x_0, r)$  $\Longleftrightarrow$  $xk\in B_d(x_0, r)$ for some $k\in K$. This means that
\begin{equation}\label{BK}
\dot x\in B_\rho(\dot x_0, r) \Longleftrightarrow x\in  B_d(x_0, r)K.
 \end{equation}
Now if $\dot a(\dot x)\ne 0$, i.~e. $a(x)\ne 0$, then  $x\in B_d'$ which  implies $xK\subseteq B_d' K$.  By \eqref{BK} this means that $\dot x\in B_\rho(\dot e, r)$.

Now we choose $c>0$ such that
 (ii) holds  for $\dot a$ (and for $B_\rho$ instead of $B$).

 The property (iii) for $\dot a$ follows from
 \begin{align*}
  \int_{G/K}\!\dot a(\dot x)d\nu(\dot x)&= \int_{G/K}\!a(x)d\nu(\dot x)\\
  &=c\!\!\int_{G/K}\!\int_K\!h(xk)d\lambda_K(k)d\nu(\dot x)\\
  &=  c\int_{G}h(x)d\lambda_G(x)=0.
  \end{align*}
So, the function $\dot a$ belongs to $H^{1,q}(G/K)$ and  $\dot a(\dot e)=a(K)=c\ne 0$.
\end{proof}

\begin{definition}\label{Kdoubl} Let $G, d, K, \lambda_G$ be as above. We call a group $(G,d)$ $K$-doubling if there is such constant $C>0$ that for all $x_0\in G$ and $r>0$ we have
\begin{equation}\label{K-doubling}
\lambda_G(B_d(x_0,2r)K)\le C \lambda_G(B_d(x_0,r)K).
 \end{equation}

\end{definition}

\begin{theorem}\label{g/k} (cf. \cite{Lie}).  Let $(\Omega,\mu)$ be a $\sigma$-compact quasi-metric space
with   positive Radon measure $\mu$ and $G$ a connected Lie group with right  invariant distance $d$ and left Haar  measure $\lambda_G$.   Assume that  a group  $(G,d)$ is $K$-doubling,  the space $H^{1,q}(G/K)$ is  separable ($q\in (1,\infty]$), and a family  $(\dot A(u))_{u\in \Omega}\subseteq \mathrm{Aut}(G/K)$ is  measurable.
If $N(\Phi,\dot A,q)<\infty $,
then a Hausdorff operator  $\mathcal{H}_{\Phi, \dot A}$ is bounded in $H^{1,q}(G/K)$ and its norm does not exceed  $N(\Phi,\dot A,q)$. 
\end{theorem}

\begin{proof} To verify the conditions of Theorem \ref{H1}  we split the proof into several steps.

Step 1. Since by our assumptions the map $u\mapsto \dot A(u)$ is measurable, and the map $\varepsilon_{\dot x}$ is continuous by Lemma \ref{eval1}, the family  $(\dot A(u))_{u\in \Omega}$ is
$\mu$-$\nu$-measurable.

Step 2.  It
 follows from \eqref{BK} that (below $1_E$ denotes the indicator of the  set $E$)
$$
\nu(B_\rho(\dot x_0, r)) =\int_{G/K}1_{B_\rho(\dot x_0, r)}(\dot x)d\nu(\dot x)=\int_{G/K}1_{B_d(x_0, r)K}(x)d\nu(\dot x).
$$
Putting in the last integral $x=yk$ where $k\in K$  we have
$$
\nu(B_\rho(\dot x_0, r)) =\int_{G/K}1_{B_d(x_0, r)K}(yk)d\nu(\dot y).
$$
On the other hand, $1_{B_d(x_0, r)K}(yk)=1_{B_d(x_0, r)K}(y)$ does not depend on $k\in K$, and therefore
$$
1_{B_d(x_0, r)K}(yk)=\int_K 1_{B_d(x_0, r)K}(yk)d\lambda_K(k).
$$
This yields in view of \eqref{weil} that
\begin{eqnarray*}
\nu(B_\rho(\dot x_0, r))&=&\int_{G/K}\left(\int_K 1_{B_d(x_0, r)K}(yk)d\lambda_K(k)\right)d\nu(\dot y)\\
&=&\int_G 1_{B_d(x_0, r)K}(y)d\lambda_G(y)=\lambda_G(B_d(x_0, r)K).\nonumber
\end{eqnarray*}

The last equality implies in particular that  $(G/K,\rho,\nu)$ is a    space of homogeneous type, since  $(G,d)$ is $K$-doubling.

Step 3. It is known  \cite[Lemma 2.6]{Lie} that every automorphism $A\in \mathrm{Aut}(G)$ is
Lipschitz with a Lipschitz constant $\|(dA)_e\|$. It follows that  each measurable family $(\dot A(u))_{u\in \Omega}\subseteq \mathrm{Aut}(G/K)$ agrees with $\rho$. Indeed, by this lemma for a  measurable function $k(u)=\|(dA(u))_e\|$ we have  
$$d(A(u)^{-1}(p),A(u)^{-1}(q))\le k(u)d(p,q)$$
 for all $p,q\in G$. Putting here $p=A(u)(x)$, $q=A(u)(y)$ we get $d(x,y)\le k(u)d(A(u)(x),A(u)(y)$ for all $x,y\in G$. This implies that 
\begin{eqnarray}\label{++}
 A(u)^{-1}(B_d(x,r))\subseteq B_d(x',k(u)r)
 \end{eqnarray}
  where $x'=A(u)^{-1}(x)$.

  In turn, the last inclusion implies that 
 $$
\dot A(u)^{-1}(B_\rho(\dot x,r))\subseteq B_\rho(\dot x',k(u)r)
 $$
because
\begin{eqnarray*}
\dot y\in \dot A(u)^{-1}(B_\rho(\dot x,r))&\Longrightarrow&  \dot A(u)\dot y\in  B_\rho(\dot x,r)\\
&\Longrightarrow& A(u)y\in  B_d(x,r)K\\
&\Longrightarrow& y\in  A(u)^{-1}(B_d(x,r))K\subseteq B_d(x',k(u)r)K\\
&\Longrightarrow&  \dot y\in  B_\rho(\dot x',k(u)r)
\end{eqnarray*}
by \eqref{BK}. Thus, $(\dot A(u))_{u\in \Omega}$ agrees with $\rho$.

Step 4. Its remains to show that  $(\dot A(u))_{u\in \Omega}$  agrees with the measure $\nu$. To this end note that the natural projection $\pi_K:G\to G/K$ is perfect (see, e.~g.,\cite[Ch. III, \S 4, Corollary 2 of Proposition 1]{gentop}).  In particular, for every compact $\dot C\subseteq G/K$ the set $C:=\pi_K^{-1}(\dot C)$ is compact in $G$. Putting $g=1_{C}$ (the indicator of $C$) in \eqref{weil} we get
\begin{eqnarray*}
\int_G 1_C(x)d\lambda_G(x)&=&\int_{G/K}\left(\int_K1_C(xk)d\lambda_K(k)\right)d\nu(\dot x)\\
&=&\int_{G/K}\left(\int_K1_{\dot C}(\dot x)d\lambda_K(k)\right)d\nu(\dot x)=\nu(\dot C).
\end{eqnarray*}
In other words, $\lambda_G(\pi_K^{-1}(\dot C))=\nu(\dot C)$ for all  compact sets $\dot C\subseteq G/K$. Since both sides of this equality are regular Borel measures, it is valid for all 
Borel sets $\dot C\subseteq G/K$ of finite measure. Therefore
\begin{eqnarray*}
\nu(\dot A(u)^{-1}(\dot C))&=&\lambda_G(\pi_K^{-1}(\dot A(u)^{-1}(\dot C)))\\
&=&\lambda_G(A(u)^{-1}(\pi_K^{-1}(\dot C)))\\
&=&\lambda_G(A(u)^{-1}(C))\\
&=&\mathrm{mod}A(u)^{-1}\lambda_G(C)\\
&=&\mathrm{mod}A(u)^{-1}\nu(\dot C)).
\end{eqnarray*}
So, since  the family  $(\dot A(u))_{u\in \Omega}$ is  measurable, it  agrees with the measure $\nu$, and  $m(\dot A(u))=\mathrm{mod}A(u)$.

Thus, all the conditions of Theorem \ref{H1} are satisfied, and the   result follows. 
\end{proof}



\subsection{The case of double coset spaces  of  Lie groups }\label{double}

In this subsection, $G$ stands for a Lie group with a Haar measure $\lambda_G$, and $K$ is a compact subgroup of  $G$ with the normalized Haar measure $\lambda_K$. We  assume that the topology of $G$ is generated by a  metric  $d$ which is $K$-bi-invariant in a sense that
$$
d(k_1xk_1',k_1yk_1')=d(x,y)
$$
for all $x,y\in G$, $k_1, k_1' \in K$.

  We denote by $\ddot x$ the element
$KxK$ ($x\in G$) of the double coset space $G/\!/K$ (the quotient  space of $G$ with respect to the equivalence relation $x\sim y:=y\in KxK$). Let $p_K:G\to G/\!/K$ stands for the natural projection. Recall that a natural topology $\tau(G/\!/K)$ by definition is the strongest topology for $p_K$ to be continuous. It is clear that $p_K$ is open.

\begin{lemma}\label{pKperfect}
The projection $p_K$ is a perfect map.
\end{lemma}

\begin{proof} First note that $p_K$ is a closed map because for each closed set $F\subseteq G$ its saturation $KFK$ with respect to the equivalence relation $x\sim y$ is closed due to the compactness of $K$ (see, e.g., \cite[Ch. I, \S  5, no. 2]{gentop12}). Since $p_K^{-1}(\ddot y)=KyK$ is compact for each $\ddot y\in G/\!/K$, the statement of the lemma follows  (see, e.g., \cite[Ch. I, \S  10, Theorem 1]{gentop12}).
\end{proof}

We equip  $G/\!/K$ with the metric
$$
\rho(\ddot x,\ddot y):=\min_{k_1k_1',k_2k_2'\in K}d(k_1xk_1',k_2yk_2')=\min_{k,k'\in K}d(x,kyk').
$$

\begin{lemma}\label{tauGK}
The function $\rho$ is a  metric in $G/\!/K$ that induces the natural topology $\tau(G/\!/K)$ of the factor space $G/\!/K$.
\end{lemma}

\begin{proof} For every $\epsilon>0$, $K$ contains elements $a,a',b,b'$ such that 
\begin{align*}
\rho(\ddot x,\ddot y)+\rho(\ddot y,\ddot z)+2\epsilon
&\ge d(x,aya')+d(y,bzb')\\
&\ge d(a^{-1}xa'^{-1}x,y)+d(y,bzb')\\
&\ge d(a^{-1}xa'^{-1}x,bzb')\\
&\ge \rho(\ddot x,\ddot z).
\end{align*}
This shows that $\rho$ satisfies the triangle inequality. It can be 
verified that $\rho$ is single-valued in $(G/\!/K)\times (G/\!/K)$, that $\rho$ is  
symmetric, and that $\rho(\ddot x, \ddot y) = 0$ if and only if $\ddot x=\ddot y$. Then 
$\rho$ is a metric. 

Further, since
$$
\rho(p_K(x),p_K(y))=\rho(\ddot x,\ddot y)\le d(x,y),
$$
the map $p_K:(G,d)\to (G/\!/K,\rho)$  is continuous. Thus, the  topology $\tau(G/\!/K)$ is stronger, then the  topology $\tau_\rho$ induced by $\rho$. 
On the other hand, $KyK\in KB_d(x,\epsilon)K$ if $\rho(\ddot x,\ddot y)<\epsilon$. 
Therefore any neighborhood of $\ddot x$ in the  topology $\tau(G/\!/K)$ includes a  neighborhood in a sense of $\tau_\rho$. 
It follows that $\tau_\rho$ is stronger, then $\tau(G/\!/K)$.
\end{proof}

Now we are going to explain what  an automorphism of a double coset space means. We proceed along the lines of the previous section.

Consider the subset
\begin{eqnarray*}
\mathrm{Aut}_K^{00}(G)&=&\{A\in \mathrm{Aut}_K(G): A(KxK)=KxK \forall x\in G\}\\
&=&\{A\in \mathrm{Aut}_K(G): A(x)\sim x \forall x\in G\}
\end{eqnarray*}
of $ \mathrm{Aut}_K(G)$.

\begin{lemma}\label{norm:closed00}
The set $\mathrm{Aut}_K^{00}(G)$ is a closed normal subgroup of  $\mathrm{Aut}_K(G)$.
\end{lemma}

\begin{proof} It is easy to verify that  $\mathrm{Aut}_K^{00}(G)$ is a subgroup of  $\mathrm{Aut}_K(G)$.

The set $\mathrm{Aut}_K^{00}(G)$ is closed, since the evaluation  map  $\alpha_{x}:\mathrm{Aut}(G) \to G$,  $\alpha_{x}(A):=A(x)$ is continuous \cite[Proposition III.3.1, p. 40]{Hoch}. Therefore if the net $A_n\to A$ in  $\mathrm{Aut}(G)$ and  $A_n\in \mathrm{Aut}_K^{00}(G)$, i.e.  $A_n(x)\in KxK  \forall x\in G$, then $A(x)\in KxK  \forall x\in G$, i.e., $A(x)\sim x \forall x\in G$ and thus  $A\in \mathrm{Aut}_K^{00}G)$, as well.

Finally, $\mathrm{Aut}_K^{00}(G)$ is an invariant subgroup of $\mathrm{Aut}_K(G)$. Indeed, for every $A\in \mathrm{Aut}_K^{00}(G)$ and every  $A_1\in \mathrm{Aut}_K(G)$
we have
$$
A_1^{-1}AA_1(KxK)=A_1^{-1}A(KA_1(x)K)=A_1^{-1}(KA_1(x)K)=KxK
$$
for all $x\in G$. So, $A_1^{-1}AA_1\in \mathrm{Aut}_K^{00}(G)$.
\end{proof} 

\begin{definition}\label{aut00}
We define the group   $ \mathrm{Aut}(G/\!/K)$ of automorphisms of  a  homogeneous space $G/\!/K$  as a factor-group $\mathrm{Aut}_K(G)/\mathrm{Aut}_K^{00}(G)$.
\end{definition}

Let  $A\in \mathrm{Aut}_K(G)$. Since
$$
A(\ddot x):= A(KxK)=\{A(k)A(x)A(k'): k, k'\in K\}=
KA(x)K=p_K(A(x)),
$$
 we get a homeomorphism   
 $$
 \ddot A:G/\!/K\to G/\!/K, \ \ddot A(\ddot x):=p_K(A(x))
 $$ 
 (see, e.g., \cite[Chapter I, \S 3, Propo. 9]{gentop12}). Also $\ddot A(o)=o$ where $o=\ddot e$.

\begin{lemma}\label{eval100}
1. $\mathrm{Aut}(G/\!/K)=\{\ddot A: A\in \mathrm{Aut}_K(G)\}.$

2. For each $x\in G$  the evaluation map 
$$
\varepsilon_{\ddot x}: \mathrm{Aut}(G/\!/K)\to G/\!/K, \varepsilon_{\ddot x}(\ddot A)= \ddot A(\ddot x)
$$
is continuous.
\end{lemma}

\begin{proof}  1. Indeed, 
$$
\mathrm{Aut}_K^{00}(G)=\{A_0\in \mathrm{Aut}_K(G):  A_0(\ddot x)=\ddot x \forall x\in G\}.
$$
Then $AA_0(\ddot x)=A(\ddot x)=\ddot A(\ddot x)$ for all $x\in G$, $A_0\in \mathrm{Aut}_K^{00}(G)$ and $A\in \mathrm{Aut}_K(G)$.
But  an arbitrary element of $\mathrm{Aut}_K(G)/\mathrm{Aut}_K^{00}(G)$  has the form $A\mathrm{Aut}_K^{00}(G)=\{AA_0: A_0\in \mathrm{Aut}_K^{00}(G)\}$  where $A\in \mathrm{Aut}_K(G)$. 
 It follows that we can identify the groups under consideration.

2. Since the map  $\alpha_{x}:\mathrm{Aut}(G) \to G$ sending $A$ onto $A(x)$ is continuous \cite[Proposition III.3.1, p. 40]{Hoch},  the map  $\beta_{x}:\mathrm{Aut}_K(G) \to G/\!/K$ sending $A$ onto $p_K A(x)=(A(x))^{{\bf\cdot \cdot}}$ is continuous, too. Let $\phi$ denotes the  natural projection $\mathrm{Aut}_K(G)\to \mathrm{Aut}(G/\!/K)$, $\phi(A)=\ddot A$. Then
 $\beta_{x}=\varepsilon_{\ddot x} \phi$ and thus the map $\varepsilon_{\ddot x} $  is continuous (see, e.g., \cite[Chapter 1, \S 3]{gentop12}).
\end{proof}

\begin{definition}\label{KKdoubl} Let $G, d, K, \lambda_G$ be as above. We call a group $(G,d)$ $KK$-doubling if there is such constant $C>0$ that for all $x_0\in G$ and $r>0$ we have
\begin{equation}\label{K-doubling}
\lambda_G(KB_d(x_0,2r)K)\le C \lambda_G(KB_d(x_0,r)K).
 \end{equation}
\end{definition}


Note \cite[Theorems 8.2B, (14.2H)]{Jew} that $X=G/\!/K$ is a double coset hypergroup (convo)  with Haar measure
$$
\nu=\int_G\delta_{p_K(x)}d\lambda_G(x)
$$ 
(the Bochner integral). The last equality means that $\nu$ is the image $p_K(\lambda_G)$ in a sense of  \cite[Ch. V, \S 6, Definition 1]{Bourb1}. In particular,
\begin{eqnarray}\label{nuE}
\nu(E)=\lambda_G(p_K^{-1}(E))
\end{eqnarray}
for all $\nu$-measurable $E\subseteq G/\!/K$ \cite[Ch. V, \S 6, Corollary1 of Proposition 2]{Bourb1}.

\begin{lemma}\label{nontriv2} Let   $G\ne K$. Then the space $H^{1,q}(G/\!/K)$ is nontrivial .
\end{lemma}

The proof of this lemma  is similar to the proof of Proposition 2 in \cite{JOTH}.

For the space $H^{1,q}(G/\!/K)$  Theorem \ref{H1}  looks as follows.

\begin{theorem}\label{g//k} (cf. \cite{JOTH}). Let $(\Omega,\mu)$ be a $\sigma$-compact quasi-metric space
with   positive Radon measure $\mu$ and $G$ a Lie group with $K-$bi-invariant distance $d$ and left Haar  measure $\lambda_G$.   Assume that  the group  $(G,d)$ is $KK$-doubling,  the space $H^{1,q}(G/\!/K)$ is  separable ($q\in (1,\infty]$), and a family  $(\ddot A(u))_{u\in \Omega}\subseteq \mathrm{Aut}(G/\!/K)$ is  measurable.
If $N(\Phi,\ddot A,q)<\infty $,
then a Hausdorff operator  $\mathcal{H}_{\Phi, \ddot A}$ is bounded in $H^{1,q}(G/\!/K)$ and its norm does not exceed  $N(\Phi,\ddot A,q)$. 
\end{theorem}

\begin{proof} To verify the conditions of Theorem \ref{H1}  we split the proof into several steps.

Step 1. Since by our assumptions the map $u\mapsto \ddot A(u)$ is measurable, and the map $\varepsilon_{\ddot x}$ is continuous by Lemma \ref{eval100}, the family  $(\ddot A(u))_{u\in \Omega}$ is $\mu$-$\nu$-measurable.

Step 2. First note that
\begin{eqnarray}\label{pKB}
p_K^{-1}(B_\rho(\ddot x_0,r))=p_K(B_d(x_0,r)).
\end{eqnarray}
Indeed,
\begin{eqnarray*}
x\in p_K^{-1}(B_\rho(\ddot x_0,r))&\Longleftrightarrow \ddot x\in B_\rho(\ddot x_0,r)\\
&\Longleftrightarrow  \rho(\ddot x,\ddot x_0)<r \Longleftrightarrow \min\limits_{k,k'\in K}d(kxk',x_0)<r\\
 &\Longleftrightarrow d(kxk',x_0)<r \mbox{  for some } k,k'\in K\\
&\Longleftrightarrow kxk'\in B_d(x_0,r) \mbox{  for some } k,k'\in K\\
 &\Longleftrightarrow x\in KB_d(x_0,r)K=p_K(B_d(x_0,r)).
\end{eqnarray*}
  It
 follows from \eqref{nuE} that 
$$
\lambda_G(KB_d(x_0,r)K)=\lambda_G(p_K(B_d(x_0,r))=\lambda_G(p_K^{-1}(B_\rho(\ddot x_0,r))=\nu(B_\rho(\ddot x_0,r)).
$$

The last equality implies in particular that  $(G/\!/K,\rho,\nu)$ is a    space of homogeneous type, since  $(G,d)$ is $KK$-doubling.

 Step 3. We are going to prove that  each measurable family $(\ddot A(u))_{u\in \Omega}\subseteq \mathrm{Aut}(G/\!/K)$ agrees with $\rho$. 
Indeed,  the inclusion \eqref{++} and the equality \eqref{pKB} imply that for each $\ddot A(u)\in \mathrm{Aut}(G/\!/K)$ we have for $x\in G$, $r>0$
$$
\ddot A(u)^{-1}(B_\rho(\ddot x,r))\subseteq B_\rho(\ddot x', k(u)r))
$$
where $k(u)=\|(dA(u))_e\|$.

In fact, by  the equality \eqref{pKB}
\begin{eqnarray*}
\ddot y\in \ddot A(u)^{-1}(B_\rho(\ddot x,r)) \Longrightarrow&\ddot A(u)(\ddot y)\in B_\rho(\ddot x,r)\\
&\Longleftrightarrow p_K(A(u)(y))\in B_\rho(\ddot x,r)\\
&\Longrightarrow A(u)(y)\in  p_K^{-1}(B_\rho(\ddot x,r))\\
&=p_K(B_d(x,r))=K(B_d(x,r)K.
\end{eqnarray*}
It follows that
\begin{eqnarray*}
y\in  A(u)^{-1}(K(B_d(x,r)K)=KA(u)^{-1}(B_d(x,r))K.
\end{eqnarray*}
Therefore  by the inclusion \eqref{++} we have
$$
\ddot y\in A(u)^{-1}(B_d(x,r))\subseteq B_\rho(\ddot x', k(u)r)).
$$
  Thus, $(\ddot A(u))_{u\in \Omega}$ agrees with $\rho$.

 Step 4. Its remains to show that  $(\ddot A(u))_{u\in \Omega}$  agrees with the Haar measure $\nu$.

Since  the natural projection $p_K:G\to G/\!/K$ is perfect, for every compact $\ddot C\subseteq G/\!/K$ the set $C:=p_K^{-1}(\ddot C)$ is compact in $G$. 
The set $C$ is saturated with respect to the equivalence relation $\sim$ and therefore the set $A(C)$ is also saturated with respect to the  relation $\sim$ for all  $A\in \mathrm{Aut}_K(G)$, too.
It follows that 
$$p_K^{-1}(\ddot A^{-1}(\ddot C))=p_K^{-1}(p_K(A^{-1}(C)))=A^{-1}(C).
$$
Now formula \eqref{nuE} yields 
\begin{eqnarray*}
\nu(\ddot A(u)^{-1}(\ddot C))&=&\lambda_G(p_K^{-1}(\ddot A(u)^{-1}(\ddot C)))\\
&=&\lambda_G(A(u)^{-1}(\pi_K^{-1}(\dot C)))\\
&=&\lambda_G(A(u)^{-1}(C))\\
&=&\mathrm{mod}A(u)^{-1}\lambda_G(C)\\
&=&\mathrm{mod}A(u)^{-1}\nu(\ddot C)).
\end{eqnarray*}
Since the both sides  of the last equality are regular Borel measures, it is valid for all Borel $\ddot C\subseteq G/\!/K$ of finite measure.
Therefore, taking into account that   the family  $(\dot A(u))_{u\in \Omega}$ is  measurable, it  agrees with the measure $\nu$, and  $m(\ddot A(u))=\mathrm{mod}A(u)$.

Thus, all the conditions of Theorem \ref{H1} are satisfied, and the   result follows. 
\end{proof}


\section{Data availability statement}
The author confirms that all data generated or analyzed during this study
are included in this article.
This work does not have any conflicts of interest.







\end{document}